\documentclass[12pt,oneside]{amsart}
\title{The accessibility problem for geometric rough differential equations}
\date{}

\author{Youness Boutaib}
\address{RWTH Aachen University\\
Chair for Mathematics of Information Processing\\ 
Pontdriesch 10\\
52062 Aachen, Germany}
\email{boutaib@mathc.rwth-aachen.de}

\usepackage{amssymb} 
\usepackage{amsfonts}
\usepackage{amsmath} 
\usepackage{relsize} 
\usepackage{amsthm} 
\usepackage[margin=1in]{geometry}  
\usepackage[usenames,dvipsnames]{xcolor} 
\usepackage{hyperref}
\usepackage{longtable}

\newtheorem{theo}{Theorem}[section]
\newtheorem{lemma}[theo]{Lemma}
\newtheorem*{example}{Example}
\newtheorem{Cor}[theo]{Corollary}
\newtheorem{Def}[theo]{Definition}
\newtheorem{Prop}[theo]{Proposition}
\newtheorem{Rem}[theo]{Remark}
\newtheorem*{Notation}{Notation}

  
\begin{document}
\maketitle

	\begin{abstract} We show how to use geometric arguments to prove that the terminal solution to a rough differential equation driven by a geometric rough path can be obtained by driving the same equation by a piecewise linear path. For this purpose, we combine some results of the seminal work of Sussmann on orbits of vector fields \cite{Sussmann} with the rough calculus on manifolds developed by Cass, Litterer and Lyons in \cite{CLL2}.
	\end{abstract}

	\section{Introduction}\label{sec:Intro}
	\subsection{Formulation of the problem}\label{subsec:IntroProblem}
	The theory of rough paths introduced by T. Lyons \cite{Lyons,CLL} (then developed in further directions by M. Gubinelli \cite{Gubinelli,Gubinelli2}) teaches us that the effects of paths on differential systems can not always be explained by the collection of values taken by the path through time. The understanding of these effects requires, in general, a convention on what is meant by the sequence of the iterated integrals of the driving path. More specifically, giving a meaning to the differential equation 
	\begin{equation}\label{eq:shortRDE}
	\mathrm{d}y_t=f(y_t)\mathrm{d}x_t
	\end{equation}
	(with some initial condition) requires not only the knowledge of the path $(x_t)_{0 \leq t \leq T}$ but also of its lift into a signature
	\begin{equation}\label{eq:shortSig}
	S(x)_{0,t}=\left(
	1,
	\int_{0\leq u_1\leq t} \mathrm{d}x_{u_1},
	\int_{0\leq u_1\leq u_2 \leq t} \mathrm{d}x_{u_1}\otimes \mathrm{d}x_{u_2},
	\int_{0\leq u_1\leq u_2 \leq u_3 \leq t} \mathrm{d}x_{u_1}\otimes \mathrm{d}x_{u_2}\otimes \mathrm{d}x_{u_3},
	\ldots
	\right).
	\end{equation}
	While the signature (\ref{eq:shortSig}) can be given a canonical meaning for any smooth path $x$ (or more precisely, for any path with finite $p$-variation for $1\leq p <2$), this is no longer the case for rougher paths like the ones frequently encountered in stochastic analysis. For example, one may interpret (\ref{eq:shortSig}) for the Brownian motion using either It\^o's or Stratonivich's integration theory and the solution to the differential equation (\ref{eq:shortRDE}) will also depend in general on this choice. What could be even less intuitive is that the lift (\ref{eq:shortSig}) for smooth paths and the solution to (\ref{eq:shortRDE}) can also be interpreted in non-canonical ways, thus describing unexpected effects. For example, taking $x$ to be the null path in $\mathbb{R}^2$, the classical theory of ordinary differential equations (ODEs) tells us that the solution $y$ to (\ref{eq:shortRDE}) is constant (equal to its initial value $y_0$). However, in the theory of rough paths, one may interpret $x$ as a rough path and choose the polynomial tensor path
	\[t\longmapsto 
	\left( 1, 
	(0,0),
	a t
	\right),
	\]
	where $a \in \mathbb{R}^2\otimes \mathbb{R}^2$ is any (antisymmetric) second order tensor in $\mathbb{R}^2$, as a definition for the second order signature of $x$. In general, the solution to (\ref{eq:shortRDE}) driven by this rough path is not going to be constant. Moreover, the solution to (\ref{eq:shortRDE}) driven by this rough lift of the null path $x$ can be given a physical meaning, for instance as the limiting behaviour of systems driven by the increasingly oscillating paths
	\[x(n):t\longmapsto \left( \frac{\cos(2 \pi n^2 t)}{n} , \frac{\sin(2 \pi n^2 t)}{n}\right) \in \mathbb{R}^2,\]
	(this now popular example can be found both in \cite{Lyons} and \cite{CLL}.) More explicitly, the sequence $(y(n))_n$ of solutions to the differential equation (\ref{eq:shortRDE}) driven by the paths $(x(n))_n$ converges (in some rough path topology) to the solution to the same equation interpreted in the rough path sense to be driven by the rough path
	\[t\longmapsto 
	\left( 1, 
	(0,0),
	\left(
	\begin{array}{cc}
	0&\pi t \\
	-\pi t &0\\
	\end{array}
	\right),
	\ldots
	\right).
	\]
	 Other examples of equations driven by the Brownian motion whose limiting behaviour is to be treated by the theory of rough paths and which can be given a physical meaning can be found in \cite{FGL}.\\
	
	In view of the above, it is legitimate to ask the following question: can rough signals through their non-canonical lifts exhibit terminal output states that cannot be observed in systems that are driven by (piecewise) smooth paths (and their canonical lifts)? More formally, for $\gamma>p\geq 1$, $\xi \in \mathbb{R}^d$, a Lipschitz-$\gamma$ map $f:\mathbb{R}^d\to \mathcal{L}(\mathbb{R}^n, \mathbb{R}^d)$ and a geometric $p$-rough path $\mathbf{X}$ in  $\mathbb{R}^n$ defined over a compact interval $[0,T]$, let $\mathbf{Y}$ be the solution to the following rough differential equation (RDE)
	\begin{equation}\label{eq:RDEIntro}
	\left\{\begin{array}{l}
	\mathrm{d}\mathbf{Y}_t=f(\mathbf{Y}_t)\mathrm{d}\mathbf{X}_t, \quad \forall\;\; 0\leq t \leq T\\
	\mathbf{Y}_0=\xi\\
	\end{array}
	\right.
	\end{equation}
	Does there exist an $\mathbb{R}^n$-valued path $x$ with bounded variation such that the solution to the classical ODE
	\begin{equation}\label{eq:ODEClassicalIntro}
	\left\{\begin{array}{l}
	\mathrm{d}y_t=f(y_t)\mathrm{d}x_t, \quad \forall\;\; 0\leq t \leq T\\
	y_0=\xi\\
	\end{array}
	\right.
	\end{equation}
	satisfies $y_T=\xi+\mathbf{Y}_{0,T}^1$? We will call this the accessibility problem.
	\subsection{Review of the literature and strategy}\label{subsec:literature}
	By the continuity of the It\^o solution map associated to (\ref{eq:RDEIntro}) and the very definition of a geometric rough path, given a geometric $p$-rough path $\mathbf{X}$ in $\mathbb{R}^n$, one can approximate, in the $p$-rough path topology and up to an arbitrarily small error, its solution $\mathbf{Y}$ by the solution $y$ (identified as its canonical signature $S(y)$) to the RDE when driven by a path of bounded variation. Consequently, one can approximate the path $(\xi+\mathbf{Y}_{0,t}^1)_{0\leq t\leq T}$ in the $p$-variation and the uniform convergence topologies by the path $(y_t)_{0\leq t\leq T}$. The question we ask here is slightly different: we are only interested in the terminal value $\xi+\mathbf{Y}_{0,T}^1$ of the output of the system and, as a price for this weaker condition, we want to reproduce this same terminal state by driving the RDE (\ref{eq:RDEIntro}) by a path of bounded variation. Indeed, in general, the path $(\xi+\mathbf{Y}_{0,t}^1)_{0\leq t\leq T}$ has only finite $p$-variation while solutions to the ODE (\ref{eq:ODEClassicalIntro}) have finite $1$-variations. Therefore, it is not always possible to replace $\mathbf{X}$ by a path of finite length and replicate the solution $(\xi+\mathbf{Y}_{0,t}^1)_{0\leq t\leq T}$ at all times.\\
	
	Let us now consider a simple case. Let $N\in\mathbb{N}^*$ and denote by $T^{(N)}(\mathbb{R}^n)$ the truncated tensor algebra of $\mathbb{R}^n$ of order $N$ defined by
	\[ T^{(N)}(\mathbb{R}^n)
	= \mathbb{R} \oplus \mathbb{R}^n 
	\oplus \left(\mathbb{R}^n\right)^{\otimes 2} 
	\oplus \cdots 
	\oplus \left(\mathbb{R}^n\right)^{\otimes N} 
	.\]
	Define the (linear) map $f: T^{(N)}(\mathbb{R}^n) \to \mathcal{L} (\mathbb{R}^n,T^{(N)}(\mathbb{R}^n))$ by
    \[\forall x \in \mathbb{R}^n ,\, 
    \forall (a_0,\ldots,a_N)\in T^{(N)}(\mathbb{R}^n):\quad 
    f(a_0,\ldots,a_N)(x)=(0,a_0\otimes x,\ldots,a_{N-1}\otimes x).\]
    Then, given a path $x:[0,T]\to \mathbb{R}^n$ of bounded variation, the ODE
    \begin{equation}\label{eq:ODESigIntro}
	\left\{\begin{array}{l}
	\mathrm{d}y_t=f(y_t)\mathrm{d}x_t, \quad \forall\;\; 0\leq t \leq T\\
	y_0=(1,0,\cdots,0) \in T^{(N)}(\mathbb{R}^n) \\
	\end{array}
	\right.
	\end{equation}
	admits as a unique solution the order-$N$ truncated signature path $\left(S_{N}(x)_{(0,t)}\right)_{0\leq t\leq T}$ of $x$ \cite{CLL, Chen}. It is also known (by using for example the shuffle product property \cite{CLL, Chen}) that the order-$N$ truncated signatures live in the free nilpotent group $G^{(N)}(\mathbb{R}^n)$ of order $N$ \cite{Reutenauer}. As $f$ is smooth, one may also consider the equation (\ref{eq:ODESigIntro}) in the rough sense. The traces of solutions to  (\ref{eq:ODESigIntro}) when driven by (geometric) rough paths also take their values in $G^{(N)}(\mathbb{R}^n)$ (for example by the continuity of the shuffle product property \cite{CLL}). The answer to the accessibility problem in the case of (\ref{eq:ODESigIntro}) is affirmative and the argument relies on one final ingredient; that the order-$N$ truncated signatures of paths of bounded variation fill the free nilpotent group $G^{(N)}(\mathbb{R}^n)$. The most known proof of this fact is a geometric one based on a theorem by Chow.
	\begin{theo}[Chow-Rashevskii]{\cite{Chow, Gromov, Rashevskii}} Let $M$ be a connected smooth manifold and let $\mathcal{D}$ be a finite family of smooth vector fields in $M$ such that these vector fields and their successive (higher order) Lie brackets span every tangent space of $M$. Then every two points in $M$ can be joined by a concatenation of integral curves of the vector fields in $\mathcal{D}$.
	\end{theo}
	First note that $G^{(N)}(\mathbb{R}^n)$ is a Lie group whose tangent space is given (through a globally defined logarithm map) by ${\mathcal{L}}^{(N)}(\mathbb{R}^n)$, the space of Lie polynomials of degree $N$ \cite{CLL, Reutenauer}. Denote by $(e_i)_{1\leq i\leq n}$ the canonical basis of $\mathbb{R}^n$ and, for every $i\in [\![1,n]\!]$, by $f^i$ the linear homomorphism defined on $T^{(N)}(\mathbb{R}^n)$ by 
	\[f^i(y)=f(y)(e_i), \quad \textrm{for all } y\in T^{(N)}(\mathbb{R}^n). \] 
	It is clear then, by the very definition of ${\mathcal{L}}^{(N)}(\mathbb{R}^n)$, that the left-invariant vector fields $f^1,\ldots, f^n$ and their successive Lie brackets span the tangent spaces of $G^{(N)}(\mathbb{R}^n)$. Hence, by Chow's theorem, any point in $G^{(N)}(\mathbb{R}^n)$ can be connected to the unit element $(1,0,\cdots,0)$ by a concatenation of integral curves of the vector fields $f^1,\ldots, f^n$. For $i\in [\![1,n]\!]$, an integral curve $\gamma$ of $f^i$ is also a solution to the ODE  
	\[\mathrm{d}\gamma_t=f(\gamma_t)\mathrm{d}x^{(i)}_t,\] 
	where $x^{(i)}: t\mapsto t e_i$ is an axis path. Therefore, a concatenation of integral curves of the vector fields $f^1,\ldots, f^n$ starting at $(1,0,\cdots,0)$ corresponds to a solution to the ODE (\ref{eq:ODESigIntro}) driven by a concatenation of axis paths, which itself is a path of bounded variation (we will revisit this argument more in detail in the proof of the main theorem of this article).  Therefore $G^{(N)}(\mathbb{R}^n)$ is indeed exactly the space of all signatures of paths of bounded variation and the accessibility problem is solved in this case.\\
	
	In the current paper, we will generalise the geometric argument above to the more generic RDE (\ref{eq:RDEIntro}) where the terminal solutions to the corresponding ODE do not define a smooth manifold with a structure as simple and as well understood as the Lie group $G^{(N)}(\mathbb{R}^n)$. There will be three main ingredients to these generalisations:
	\begin{itemize}
	\item A thorough understanding of Lipschitz vector fields and their flows (Section \ref{sec:LipFlows}). This section recalls several of the results in \cite{Boutaib} and provides new ones that will be the base of all the proofs of Lipschitz regularity in this paper.
	\item A geometric structure on which one can define rough paths and RDEs (Sections \ref{sec:LipGeo} and \ref{sec:RoughOnManifolds}). In Section \ref{sec:LipGeo}, we recall the definition of Lipschitz manifolds in \cite{CLL2}, show how to endow the unit ball (and sets that are diffeomorphic to it) with such a structure and explain how the Lipschitz regularity of vector fields translate into the Lipschitz regularity of the transport map they induce between Lipschitz manifolds. Section \ref{sec:RoughOnManifolds} recalls the main elements of the theory of rough paths on manifolds as laid out in \cite{CLL2}.
	\item Sussmann's seminal work on orbits of vector fields \cite{Sussmann}, which is a more flexible and more general version of Chow's theorem (Section \ref{sec:LipOrbits}). We will adapt these results to our needs and build local Lipschitz structures on neighborhoods contained in the same orbit.
	\end{itemize}
	In Section \ref{sec:MainThm}, we will combine all our constructions in the preceding sections to prove the general accessibility property for RDEs driven by geometric rough paths. We will conclude the paper with some final remarks in Section \ref{sec:Conclusion}.\\
	
	We end this introduction with a few words on some prerequisites, conventions and notations. First, and due to it becoming a far more popular and understood topic nowadays than a few years ago, some knowledge of the classical theory of (geometric) rough paths is assumed. Classical rough paths in Banach spaces are denoted in bold letters, with components denoted as follows
	\[
	\mathbf{X}_{s,t}=\left(
	1, \mathbf{X}^1_{s,t}, \mathbf{X}^2_{s,t}, \ldots
	\right).
	\]
	We simply denote the null rough path by $0$
	\[
	0_{s,t}=\left(
	1, 0, 0, \ldots, 0,\ldots
	\right).
	\]
	Second, most of the maps that will be considered in this paper are defined between vector spaces. Consequently, most of the statements below can be expressed in terms of directional or Fr\'echet derivatives. However, as we aim to build a non-flat geometric structure, the results and the intuition are more straightforward when sometimes expressed in a geometric language (e.g. pushforwards). The standard characterisations of the tangent spaces to finite dimensio\-nal vector spaces will be heavily used. For example, for a smooth map $f:\mathbb{R}^n\to \mathbb{R}^p$ and $x,v \in \mathbb{R}^n$, $f_*(x)(v)$ is sometimes used instead of writing $f_*(x)(\mathrm{D}_{x,v})$ ($\mathrm{D}_{x,v}$ being the derivation (tangent vector) at $x$ in the direction $v$), is identified with $\mathrm{d}f(x)(v)$ and can then be understood either as a vector in $\mathbb{R}^p$ or a tangent vector to $\mathbb{R}^p$ at $f(x)$. Finally, we will heavily rely on the following notations throughout this article (some of which have already been used in this introduction). \\ 
	
	\begin{longtable}{llp{0.81\linewidth}}
	$\mathcal{L}(E,F)$&:& The space of (continuous) linear mappings from a normed vector space $E$ to a normed vector space $F$.\\
	$E^{\otimes n}$&:& The space of homogeneous tensors of the vector space $E$ of order $n$ for $n\in\mathbb{N}^*$.\\	
	$B_n(x,\alpha)$&:& The open ball centred at $x\in\mathbb{R}^n$ of radius $\alpha$, $\mathbb{R}^n$ being endowed with the Euclidean norm.\\ 
	$\pi_{n,p,d}$&:& The orthogonal projection defined from $\mathbb{R}^{n}\simeq \mathbb{R}^{p}\oplus \mathbb{R}^{n-p}$ into $\mathbb{R}^{d}\simeq\mathbb{R}^{p}\oplus \mathbb{R}^{d-p}$ whose kernel is $\{0_{\mathbb{R}^{p}}\}\oplus \mathbb{R}^{n-p}$ and image is $ \mathbb{R}^{p} \oplus \{0_{\mathbb{R}^{d-p}}\}$.\\
	$\mathrm{Id}_U$&:& The identity map on the set $U$.\\
	$T_xM$&:& The tangent space to a manifold $M$ at $x\in M$.\\
	$\mathcal{\tau}(M)$&:& The set of vector fields on a manifold $M$.\\
	$[p]$&:& The integer part of a real number $p$, i.e. the only integer such that $0\leq p-[p]<1$.\\
	$\overline{A}$&:& The closure of a subset $A$ of a topological space.\\	
\end{longtable}	
	\section{A review of results on Lipschitz maps and their flows}\label{sec:LipFlows}
	\subsection{Lipschitz maps}
	We first recall the classical notion of Lipschitz regularity (as presented for example in \cite{CLL}). Given two Banach spaces $E$ and $F$ and an integer $k\geq 1$, we denote by $\mathcal{L}_s(E^{\otimes k},F)$ the space of symmetric $k$-linear mappings from $E$ to $F$. We will use, without ambiguity, the same notation $\|.\|$ to designate a norm on $E^{\otimes k}$ and the norm on $F$. We will also avoid the discussion on the required properties from the norms on tensor products of spaces and instead refer the interested reader either to the literature on rough paths (e.g. \cite{CLL}) or the manuscript \cite{Boutaib}.
	\begin{Def}\label{Def:LipMap} Let $n\in \mathbb{N}$ and $0< \varepsilon\leq 1$. Let $E$ and $F$ be two normed vector spaces and $U$ be a subset of $E$. Consider the maps $f^0:U\to E$, and, for every $k \in [\![1,n]\!]$, $f^k:U\to \mathcal{L}_s(E^{\otimes k},F)$.	For $k \in [\![0,n]\!]$, the map $R_k:E\times E\to \mathcal{L}(E^{\otimes k},F)$ defined by
	\[\forall x,y \in U, \forall v \in E^{\otimes k}: f^k(x)(v)=\sum\limits_{j=k}^{n} f^j(y)\left(\frac{v\otimes (x-y)^{\otimes (j-k)}}{(j-k)!}\right)+R_k(x,y)(v)\]
	is called the remainder of order $k$ associated to $f=(f^0,f^1,\ldots,f^n)$.\\	
	The collection $f=(f^0,f^1,\ldots,f^n)$ is said to be Lipschitz of degree $n+\varepsilon$ on $U$ (or in short a $\textrm{Lip}-(n+\varepsilon)$ map) if there exists a constant $M$ such that for all $k \in [\![0,n]\!]$, $x,y \in U$ and $v_1,\ldots,v_k \in E$:
	\begin{enumerate}
	\item $\|f^k(x)(v_1\otimes\cdots\otimes v_k)\| \leq M \|v_1\otimes\cdots\otimes v_k\| $;
	\item $\|R_k(x,y)(v_1\otimes\cdots\otimes v_k)\|\leq M \|x-y\|^{n+\varepsilon-k}\|v_1\otimes\cdots\otimes v_k\|$.
	\end{enumerate}
     The smallest constant $M$ for which the properties above hold is called the $\textrm{Lip}-(n+\varepsilon)$-norm of $f$ and is denoted by $\|f\|_{\textrm{Lip}-(n+\varepsilon)}$.
     \end{Def}
    Note that Definition \ref{Def:LipMap} has the advantage of being purely quantitative: no topological assumptions are made on the domain $U$. 
    \begin{Notation}
    For $\gamma>0$, we will denote by $\lceil\gamma\rceil$ the unique integer such that $0<\gamma- \lceil\gamma\rceil\leq 1$.
    \end{Notation}
    The following proposition shows that Lipschitz regularity is a uniform local property.
	\begin{Prop}\label{Prop:LocalLipGen}\cite{Boutaib} Let $\gamma>0$. Let $E$ and $F$ be two normed vector spaces and $U$ be a subset of $E$. We assume that there exists $\delta>0$ and $M\geq0$ such that, for every $x\in U$, $f_{|B(x,\delta)\cap U}$ is $\textrm{Lip}-\gamma$ with a norm less than or equal to $M$ (where $f=(f^0,\ldots,f^{\lceil\gamma\rceil}))$. Then $f$ is $\textrm{Lip}-\gamma$ and there exists a constant $C_{\delta,\gamma}$ (depending only on $\delta$ and $\gamma$) such that
	\[\|f\|_{\textrm{Lip}-\gamma}\leq C_{\delta,\gamma} M.\] 
	\end{Prop}
	A very useful characterisation of Lipschitz maps can be obtained in the case when the domain is assumed to be open. The proof of the following basic result can be found for example in \cite{CLL2}.
	\begin{theo}\label{thm:LipOpenCvx} Let $n\in \mathbb{N}$, $0< \varepsilon\leq 1$ and $M\geq 0$. Let $E$ and $F$ be two normed vector spaces and $U$ be a subset of $E$.  Let $f^0:U\to F$ and, for every $k\in [\![1,n]\!]$, $f^k:U\to \mathcal{L}(E^{\otimes k},F)$ be maps. We consider the two following assertions
	\begin{description}
	\item[(A1)] $(f^0,f^1,\ldots,f^n)$ is $\textrm{Lip}-(n+\varepsilon)$ and $\|f^0\|_{\textrm{Lip}-(n+\varepsilon)}\leq M$.
	\item[(A2)] $f^0$ is $n$ times differentiable, with $f^1,\ldots,f^n$ being its successive derivatives. Moreover, $\|f^0\|_{\infty}$, $\ldots$, $\|f^n\|_{\infty}$ are upper-bounded by $M$ and for all $x,y\in U$, one has
	\[ \|f^n(x)-f^n(y)\|\leq M \|x-y\|^{\varepsilon}.\]
	\end{description}
	If $U$ is open then $\mathbf{(A1)}\Rightarrow \mathbf{(A2)}$. If $U$ is open and convex then $\mathbf{(A1)}\Leftrightarrow \mathbf{(A2)}$.
	\end{theo}
	\begin{Rem} Following Theorem \ref{thm:LipOpenCvx}, the maps $f^1,\ldots,f^n$ are uniquely determined by $f^0$ if the domain of definition is open. While we stress that this is not the case in general, we will often use the shorthand statement that ``$f^0$ is Lipschitz'' instead of ``$(f^0,\ldots,f^n)$ is Lipschitz'', assuming that there is no confusion as to the knowledge of $f^1,\ldots,f^n$.
	\end{Rem}
	There exists a natural embedding between spaces of Lipschitz maps.
	\begin{theo}\label{EmbedLip}\cite{Boutaib} Let $\gamma,\gamma'>0$ such that $\gamma'\leq \gamma$. Let $E$ and $F$ be two normed vector spaces and $U$ be a subset of $E$. Let $f:U\to F$ be a $\textrm{Lip}-\gamma$ map. Then $f$ is $\textrm{Lip}-\gamma'$ and there exists a constant $M_{\gamma,\gamma'}$, depending only on $\gamma$ and $\gamma'$, such that $\|f\|_{\textrm{Lip}-\gamma'}\leq M_{\gamma,\gamma'} \|f\|_{\textrm{Lip}-\gamma}$.
	\end{theo}
	As an illustration of the above results, we construct a Lipschitz bijection from the Euclidean space to the unit ball, which will be of use to us in the subsequent sections.
	\begin{lemma}\label{lemma:GlobalLipDiff} Let $d\in \mathbb{N}^*$. Let $F:\mathbb{R}^d \to B_d(0,1)$ be the map given by
	\[F:y\longmapsto \frac{y}{\sqrt{1+\|y\|^2}}.\]
	Then $F$ is a homeomorphism. Moreover:
	\begin{itemize}
		\item $F$ is Lipschitz of any degree,
		\item For every $\delta \in (0,1)$, its inverse $F^{-1}_{|B_d(0,\delta)}$ is Lipschitz of any degree.
	\end{itemize}
	\end{lemma}
	\begin{proof} $F$ is clearly a homeomorphism and is classically smooth with bounded derivatives at any order. For $n\in\mathbb{N}$ and $\varepsilon\in (0,1)$ and for every $x \in \mathbb{R}^d$, as $F^{(n+1)}$ is bounded on $B_d(x,1)$ by a universal constant then $F^{(n)}$ is $\varepsilon$-H\"older on $B_d(x,1)$ with a H\"older constant that is independent of $x$. Hence, by Theorem \ref{thm:LipOpenCvx}, $F$ is Lip-$(n+\varepsilon)$ on $B_d(x,1)$ with a Lipschitz norm that is independent of $x$. We deduce then by Proposition \ref{Prop:LocalLipGen} that $F$ is Lip-$(n+\varepsilon)$.\\
	The inverse of $F$ is given by 
	\[F^{-1}:y\mapsto \frac{y}{\sqrt{1-\|y\|^2}}.\] 
	The same previous arguments apply to $F^{-1}$ when restricted to $B_d(0,\delta)$ for $\delta \in (0,1)$ (the Lipschitz constants will explode as $\delta$ nears 1.) This concludes this proof.
	\end{proof}
	In order to obtain sharper quantitative results or to include important functions (for example linear maps) for which one of the axioms of Lipschitz regularity fails, a notion of almost Lipschitz regularity has been introduced in \cite{Boutaib}. We reintroduce it here to allow for a certain flexibility in our proofs.
	\begin{Def}\cite{Boutaib} Let $n\in \mathbb{N}^*$, $0< \varepsilon\leq 1$ and $\delta \in (0,\infty)\cup \{\infty\}$. Let $E$ and $F$ be two normed vector spaces and $U$ be a subset of $E$. Consider the maps $f^0:U\to E$, and, for every $k \in [\![1,n]\!]$, $f^k:U\to \mathcal{L}_s(E^{\otimes k},F)$. Denote by $R_0:U\times U\to E$ the remainder map of order $0$ associated to $f=(f^0,f^1,\ldots,f^n)$.	The collection $f$ is said to be almost Lipschitz of degree $n+\varepsilon$ on domains of size $\delta$ of $U$ if $(f^1,\ldots,f^n)$ is $\textrm{Lip}-(n+\varepsilon-1)$ and there exists a non-negative constant $M$, such that
	\[\forall x,y \in U:\quad \|x-y\|< \delta \Rightarrow \|R_0(x,y)\|\leq M \|x-y\|^{n+\varepsilon}.\]
	If $\|R_0\|_{\infty,\delta}$ denotes the smallest value for such a constant $M$, we will denote
	\[\|f\|_{\delta,\textrm{Lip}-(n+\varepsilon)}=\max(\|f^1\|_{\textrm{Lip}-(n+\varepsilon-1)} , \|R_0\|_{\infty,\delta}).\]
	When $\delta$ is infinite, we will merely say that $f$ is almost Lip-$(n+\varepsilon)$ on $U$.
	\end{Def}
	On open convex sets, there exists a simple criterion for identifying almost Lipschitz maps.
	\begin{lemma}\cite{Boutaib}\label{AlmostLipDiffMaps} Let $n\in \mathbb{N}^*$ and $0< \varepsilon\leq 1$. Let $E$ and $F$ be two normed vector spaces and $U$ be an open convex subset of $E$. Let $f:U\to F$ be an $n$-times continuously differentiable map with successive derivatives respectively denoted $f^1,\ldots,f^n$. Then $f$ is almost Lipschitz of degree $n+\varepsilon$ on $U$ if and only if $f^1$ is Lipschitz of degree $n+\varepsilon-1$ on $U$. In this case
	\[\|f\|_{\infty,\textrm{Lip}-(n+\varepsilon)}=\|f^1\|_{\textrm{Lip}-(n+\varepsilon-1)}.\]
	\end{lemma}
	An example of a consequence of working with almost Lipschitz maps is the following general composition theorem.
	\begin{theo}\label{thm:CompoLip}\cite{Boutaib} Let $E$, $F$ and $G$ be three normed vector spaces. Let $U$ be a subset of $E$ and $V$ be a subset of $F$. Let $\delta>0$ and $\gamma\geq1$. Let $f:U \to F$ be an almost $\textrm{Lip}-\gamma$ map on domains of size $\delta$ of $U$ and $g:V \to G$ be a $\textrm{Lip}-\gamma$ map such that $f(U) \subseteq V$. Then $g \circ f$ is $\textrm{Lip}-\gamma$ and there exists a constant $C_{\delta, \gamma}$ (depending only on $\delta$ and $\gamma$) such that
	\[\|g\circ f\|_{\textrm{Lip}-\gamma} \leq C_{\delta, \gamma}\|g\|_{\textrm{Lip}-\gamma}\max( \|f\|_{\delta,\textrm{Lip}-\gamma}^{\gamma},1).\]
	\end{theo}
	\begin{Rem} In the above theorem, if $U$ and $V$ are open, then the collection $g \circ f$ is defined by taking $g \circ f$ and its higher derivatives. Otherwise, said collection will consist of $g \circ f$ and ``formal higher derivatives'' obtained by applying the chain rule using the collections $g$ and $f$ (see \cite{Boutaib} for more details.)
	\end{Rem}
	\begin{Rem} In the context of Theorem \ref{thm:CompoLip}, if $f$ is assumed to be $\textrm{Lip}-\gamma$, then there exists a constant $C_\gamma$ depending only on $\gamma$ such that
	\[\|g\circ f\|_{\textrm{Lip}-\gamma} \leq C_{\gamma}\|g\|_{\textrm{Lip}-\gamma}\max( \|f\|_{\textrm{Lip}-(n+\varepsilon)}^{\gamma},1).\]
	\end{Rem}
	A simple case of composition is where one of the maps is linear, although, technically, a continuous linear map defined on the whole space is not Lipschitz in general (but is almost Lipschitz).
	\begin{Prop}\label{CompoLinFunc}\cite{Boutaib} Let $\gamma>0$ and $E$, $F$ and $G$ be three normed vector spaces.
	\begin{itemize}
		\item Let $U$ be a subset of $E$, $f:U \to F$ be a $\textrm{Lip}-\gamma$ map and $u: F\to G$ a bounded linear map. Then $u\circ f$ is $\textrm{Lip}-\gamma$ and $\|u\circ f\|_{\textrm{Lip}-\gamma} \leq \|u\|\|f\|_{\textrm{Lip}-\gamma}$.
		\item Let $v: E\to F$ a bounded linear map and $g:F \to G$ be a $\textrm{Lip}-\gamma$ map. Then $g\circ v$ is $\textrm{Lip}-\gamma$ and $\|g\circ v\|_{\textrm{Lip}-\gamma} \leq \|g\|_{\textrm{Lip}-\gamma} \max(1,\|v\|^{\gamma})$.
	\end{itemize}
	\end{Prop}
	We have similar results for the image of a pair of Lipschitz maps by a bilinear map.
	\begin{Prop}\label{Compobifunc}\cite{Boutaib} Let $\gamma>0$. Let $E$, $F$, $G$ and $H$ be normed vector spaces and $U$ be a subset of $E$. Let $f:U \to F$ and $g:U \to G$ be two $\textrm{Lip}-\gamma$ maps. Let $B: F\times G\to H$ be a continuous bilinear map. Then $B(f,g):U \to H$ is $\textrm{Lip}-\gamma$ and there exists a constant $C_\gamma$ such that
	\[\|B(f,g)\|_{\textrm{Lip}-\gamma} \leq C_\gamma\|B\|\|f\|_{\textrm{Lip}-\gamma}\|g\|_{\textrm{Lip}-\gamma}.\]
	\end{Prop}
	We will mostly use Proposition \ref{Compobifunc} for composition maps of the type $x\mapsto g(x)\circ f(x)$, with $f$ and $g$ being Lipschitz maps taking values in appropriate spaces of linear maps. An example of such result is the following lemma about the Lipschitz regularity of pullbacks of one-forms, which will be of use to us in a subsequent section.
	\begin{lemma}\label{lemma:PullBackOneForm} Let $\gamma> 1$ and $\delta>0$. Let $E$, $F$ and $G$ be normed vector spaces and $U$ be a subset of $E$. Let $f:U \to F$ be an almost Lip-$\gamma$ map on domains of size $\delta$ of $U$ and $\alpha:F\to \mathcal{L}(F,G)$ be a Lip-$(\gamma-1)$ one-form. Then $f^*\alpha$ is a Lip-$(\gamma-1)$ one-form and there exists a constant $C_{\delta, \gamma}$ such that
	\[\|f^*\alpha\|_{\textrm{Lip}-(\gamma-1)} \leq
	C_{\delta, \gamma}\|\alpha\|_{\textrm{Lip}-(\gamma-1)}
	\|f_*\|_{\textrm{Lip}-(\gamma-1)}
	\max( \|f\|_{\delta,\textrm{Lip}-\gamma}^{\gamma-1},1)
	.\]
	\end{lemma}
	\begin{proof}
	By the composition theorem \ref{thm:CompoLip}, $\alpha\circ f$ is Lip-$(\gamma-1)$ and there exists a constant $C_{\delta,\gamma}$ such that 
	\[\|\alpha\circ f\|_{\textrm{Lip}-(\gamma-1)} \leq C_{\delta, \gamma}\|\alpha\|_{\textrm{Lip}-(\gamma-1)}\max( \|f\|_{\delta,\textrm{Lip}-\gamma}^{\gamma-1},1).\]
	By the definition of almost Lipschitz regularity, $f_*$ is Lip-$(\gamma-1)$. Considering the bilinear map
	\[B: (u,v)\in \mathcal{L}(E,F)\times \mathcal{L}(F,G)
	\longmapsto
	v \circ u \in \mathcal{L}(E,G),\]
	we conclude, by applying Proposition \ref{Compobifunc}, that $f^*\alpha=B(f_*,\alpha\circ f)$ is Lip-$(\gamma-1)$ with the desired bound on its Lip-$(\gamma-1)$ norm.
	\end{proof}
	In the rest of this paper, Lipschitz diffeomorphisms will be key in transporting Lipschitz structures.
	\begin{Def} Let $\gamma > 0$. Let $E$ and $F$ be two normed vector spaces, $U$ be a subset of E and $V$ a subset of $F$. A map $f:U\to V$ is said to be a $\textrm{Lip}-\gamma$ diffeomorphism if $f$ is $\textrm{Lip}-\gamma$ and bijective and $f^{-1}$ is also $\textrm{Lip}-\gamma$. We define in a similar way almost $\textrm{Lip}-\gamma$ diffeomorphisms.
	\end{Def}
	One of the central arguments in building the charts and thus the manifold structure in \cite{Sussmann} is the constant rank theorem. To reproduce this argument in our setting, we will need a Lipschitz version of said theorem. The following is a simplified statement of the original Lipschitz constant rank theorem that appeared in \cite{Boutaib}. It can be obtained from the original version by simple arguments of restricting the maps to bounded sets and rescaling.
	\begin{theo}[Constant Rank]\label{ConstantRankTheo} {\cite{Boutaib}}
	Let $\gamma>1$ and $n, d, p \in \mathbb{N}^*$. Let $U$ be an open subset of $\mathbb{R}^n$ and $\varphi:U\to \mathbb{R}^d$ be a Lip-$\gamma$ map (whose derivative at each point is) of rank at most $p$. Let $x_0 \in U$ such that $\mathrm{d}\varphi(x_0)$ is of rank $p$. Then, there exist
	\begin{itemize}
	\item a $\textrm{Lip}-\gamma$ diffeomorphism $f:U_0\to B_n(0,1)$ defined on an open subset $U_0$ of $\mathbb{R}^n$ containing $x_0$ and such that $f(x_0)=0$ and,
	\item a $\textrm{Lip}-\gamma$ diffeomorphism $g:W\to B_d(0,1)$ defined on an open subset $W$ of $\mathbb{R}^d$ containing $\varphi(x_0)$ and such that $g(\varphi(x_0))=0$,
	\end{itemize}
	such that, for all $(x_1,\ldots,x_p)\in B_n(0,1)$
	\[g\circ \varphi \circ f^{-1}(x_1,\ldots,x_n)=(x_1,\ldots,x_p,0,\ldots,0) \in B_d(0,1),\]
	i.e., the following diagram commutes
		\[\begin{array}{ccc}
	 U_0(\subseteq \mathbb{R}^n)&\overset{\varphi}\longrightarrow&(\varphi(U_0)\subseteq)W(\subseteq \mathbb{R}^d)\\
	f\Big\downarrow&&\Big\downarrow g\\
	B_n(0,1)&\overset{\pi_{n,p,d}}\longrightarrow& B_d(0,1)\\
	\end{array}\]
	\end{theo}
	\subsection{Flows of Lipschitz vector fields}
	We first recall the notions of integral curves and flows. The existence and uniqueness results are very classical in the case of $\mathcal{C}^1$ vector fields and trivial for Lip-$1$ ones using for example Picard-Lindel\"of's theorem. 
	\begin{Def} Let $d\in\mathbb{N}^*$ and $f$ be a Lip-$1$ vector field on $\mathbb{R}^d$. For $y \in \mathbb{R}^d$, we call integral curve of $f$ starting at $y$ the unique solution $\gamma_y$ to the differential equation
	\[
	\left\{\begin{array}{l}
	\mathrm{d}\gamma_y(t)=f(\gamma_y(t))\mathrm{d}t, \quad \forall  t \in \mathbb{R},\\
	\gamma_y(0)=y.\\
	\end{array}
	\right.
	\]
	The restriction of an integral curve of $f$ to an arbitrary interval will still be called an integral curve of $f$. We call flow of the vector field $f$ and denote $\tilde{f}$ the map
	\[\tilde{f}: (t,y)\in \mathbb{R}\times \mathbb{R}^d 
	\longmapsto \gamma_y(t) \in \mathbb{R}^d .
	\]
	We will denote by $\tilde{f}_t$, for $t\in\mathbb{R}$, the map $y\mapsto \tilde{f}(t,y)$; and by $\tilde{f}_y$, for $y\in \mathbb{R}^d$, the integral curve $t\mapsto \tilde{f}(t,y)$.
	\end{Def}
	We will need quantitative results on the differentiability and Lipschitz regularity of flows. The analogous qualitative results for smooth vector fields are classical.
	\begin{theo}{\cite{Boutaib}}\label{theo:LipRegFlow}
	Let $d \in \mathbb{N}^*$ and $\gamma>1$. Let $f$ be a $\textrm{Lip}-\gamma$ vector field on $\mathbb{R}^d$. Then
	\begin{enumerate}
	\item $\forall T,r\in\mathbb{R}^*_+, \forall y\in \mathbb{R}^d:
	\quad\tilde{f}((-T,T)\times B_d(y,r))\subseteq B_d(y,r+T\|f\|_{\textrm{Lip}-1})$,
	\item $\tilde{f}$ is almost $\textrm{Lip}-\gamma$ on $(-T,T)\times \mathbb{R}^d$ and there exists a constant $C$ depending only on $T$, $\gamma$ and $\|f\|_{\textrm{Lip}-\gamma}$ such that $\|\mathrm{d}\tilde{f}\|_{\textrm{Lip}-(\gamma-1)}\leq C$.
	\end{enumerate}
	\end{theo}
	\section{Lipschitz geometry}\label{sec:LipGeo}
	\subsection{Lipschitz Manifolds}\label{subsec:LipManifolds}
	We now recall and build on the definition of Lipschitz manifolds as laid out in \cite{CLL2}.
	\begin{Def}\cite{CLL2}\label{Def:LipManCLL2} Let $\gamma>0$. Let $n\in\mathbb{N}^*$ and let $M$ be an $n$-topological manifold. Let $I$ be a countable set and, for every $i\in I$, $U_i$ be an open subset of $M$ and $\phi_i:M\to\mathbb{R}^n$ be a compactly supported map such that its restriction on $U_i$ defines a homeomorphism. We say that $M$ is a Lipschitz-${\gamma}$ manifold with the Lipschitz-${\gamma}$ atlas $((U_i, \phi_i))_{i\in I}$ if the following properties are satisfied
	\begin{itemize}\item $(U_i)_{i\in I}$ is a pre-compact locally finite cover of $M$;
	\item For every $i\in I$, $\phi_i(U_i)=B_n(0,1)$;
	\item There exists $\delta\in(0,1)$, such that $(U_i^{\delta})_{i\in I}$ covers $M$, where, for every $i\in I$,
	\[U_i^{\delta}=\left({\phi_i}_{|U_i}\right)^{-1}(B_n(0,1-\delta));\]
	\item There exists $L>0$, such that, for every $i,j\in I$, $\phi_j\circ{({\phi_i}_{|U_i})}^{-1}:B_n(0,1)\to\mathbb{R}^n$ is Lipschitz-$\gamma$ and $\|\phi_j\circ{({\phi_i}_{|U_i})}^{-1}\|_{\textrm{Lip}-\gamma}\leq L$.
	\end{itemize}
	\end{Def}
	Finite dimensional vector spaces are naturally endowed with Lipschitz structures \cite{CLL2}. In the remaining of this paper, we will explicitly refer to the structure defined by the following proposition.
	\begin{Prop}\label{Prop:EuclidLipManifoldEuclid} Let $\gamma\geq 1$ and $n\in\mathbb{N}^*$. Denote by $(e_1,\ldots,e_n)$ the canonical basis of $\mathbb{R}^n$. Let $\varphi$ be a Lip-$\gamma$ extension on $\mathbb{R}^n$ of $\mathrm{Id}_{B_n(0,1)}$ with support in $B_n(0,2)$. For $x\in \mathbb{R}^n$, let $\varphi_x$ be the map $y\mapsto \varphi(y-x)$. Finally, let $I$ be the countable subset of $\mathbb{R}^n$ defined by
	\[I=\left\{\sum_{i=1}^{n}\frac{k_i}{\sqrt{n}} e_i,\;\; k_1,\ldots,k_n \in \mathbb{Z}\right\}.\]
	Then  $(B_n(x,1),\varphi_x)_{x\in I}$ is a Lip-$\gamma$ atlas on $\mathbb{R}^n$.
	\end{Prop}
	\begin{Rem}
	It is straightforward to endow the Cartesian product of two Lipschitz manifolds (with the same regularity) with a natural Lipschitz structure (c.f. \cite{CLL2}).
	\end{Rem}
	Lipschitz structures can be canonically transported by homeomorphisms, irrespective of their regularity.
	\begin{Prop}\label{Prop:TransportLipStructure} Let $\gamma\geq 1$. Let $M$ and $N$ be two topological spaces. Let $F: M\to N$ be a homeomorphism. We assume that $M$ is a Lip-$\gamma$ manifold. Then $N$ has a natural Lip-$\gamma$ structure induced by $F$.
	\end{Prop}
	\begin{proof}
	Let $(U_i, \varphi_i)_{i\in I}$ be Lip-$\gamma$ atlas on $M$. For each $i\in I$, we define $V_i=F(U_i)$ and $\psi_i = \phi_i \circ F^{-1}$. It is then trivial to show that $(V_i, \psi_i)_{i\in I}$ defines a Lip-$\gamma$ atlas on $N$.
	\end{proof}
	\begin{Rem}
	In the context of Proposition \ref{Prop:TransportLipStructure}, the local representation of the function $F$ in this induced structure is the identity function, which is obviously smooth (and Lipschitz when restricted to a bounded domain).
	\end{Rem}
	Recall that in the context of Definition \ref{Def:LipManCLL2}, a single chart does not constitute a Lipschitz atlas. Consequently, one cannot construct a Lipschitz structure on the unit ball using only the identity map. In this regard, Proposition \ref{Prop:TransportLipStructure} provides us with  a simple way to show that the unit ball has indeed a Lipschitz structure. We will see in the next subsection that this structure is also natural.
	\begin{Cor}\label{UnitBallLip} For every $\gamma\geq 1$, the unit ball in the Euclidean space has a structure of a Lip-$\gamma$ manifold.
	\end{Cor}
	\begin{proof}
	It suffices to consider the homeomorphism of Lemma \ref{lemma:GlobalLipDiff} in the context of Proposition \ref{Prop:TransportLipStructure}. 
	\end{proof}
	\subsection{Lipschitz regularity on manifolds}\label{subsec:LipManifoldMaps}
	We introduce now the notion of Lipschitz regula\-rity of maps and one-forms on Lipschitz manifolds.
	\begin{Def}\cite{CLL2} Let $\gamma_0 \geq\gamma>0$ and $n\in\mathbb{N}^*$. Let $M$ be a Lip-$\gamma_0$ $n$-dimensional manifold and $((U_i, \phi_i))_{i\in I}$ be a Lip-$\gamma_0$ atlas on $M$. Let $V$ be a Banach space.
	\begin{itemize}
	\item We say that a map $f:M\to V$ is Lip-$\gamma$ if there exists a non-negative constant $C$, such that, for every $i\in I$, $f\circ{({\phi_i}_{|U_i})}^{-1}:B_n(0,1)\to V$ is Lip-$\gamma$ with a norm less than or equal to $C$. The smallest such constant $C$ is then called the Lip-$\gamma$ norm of $f$ and is denoted by $\|f\|_{\textrm{Lip}-\gamma}$.
	\item Assume that $\gamma\in(0,\gamma_0-1]$. We say that a $V$-valued one-form $\alpha$ is Lip-$\gamma$ if there exists a non-negative constant $C$, such that, for every $i\in I$ the pullback
	\[({\phi_i}_{|U_i}^{-1})^*\alpha: B_n(0,1)\to \mathcal{L}(\mathbb{R}^n,V)\]
	is a Lip-$\gamma$ one-form with a norm less than or equal to $C$. The smallest such constant $C$ is then called the Lip-$\gamma$ norm of $\alpha$ and is denoted by $\|\alpha\|_{\textrm{Lip}-\gamma}$.
	\end{itemize}	 
	\end{Def}
	\begin{Rem} It is easy to check that the classical and manifold notions of Lipschitz regularity are the same on finite-dimensional vector spaces (endowed with their natural Lipschitz structure of Proposition \ref{Prop:EuclidLipManifoldEuclid}).
	\end{Rem}
	\begin{Rem} It is straightforward to generalise the above definition of Lipschitz regularity on manifolds to maps and one-forms that are not globally defined. In this case, one checks only the Lipschitz regularity (of the local representations) on (images of) the intersections of the domain of definition with the chart domains (c.f. \cite{CLL2}).
	\end{Rem}
	The following lemma is going to be of technical use to us later.
	\begin{lemma}\label{lemma:InvChartLip} Let $p\in \mathbb{N}^*$ and $\gamma\geq 1$. Let $V$ be a subset of a normed vector space. We assume that there exists a homeomorphism $F:\mathbb{R}^p\to V$ that is Lip-$\gamma$. Then the inverses of the chart maps in the $F$-induced Lip-$\gamma$ atlas on $V$ are Lip-$\gamma$ in the classical sense with a uniform upper bound on their Lip-$\gamma$ norms.
	\end{lemma}
	\begin{proof}
	If $(B_p(x,1),\phi_x)_{x\in I}$ denotes a Lip-$\gamma$ atlas on $\mathbb{R}^p$ (Proposition \ref{Prop:EuclidLipManifoldEuclid}), then the corresponding atlas on $V$ is given by $(F(B_p(x,1)),\phi_x\circ F^{-1})_{x\in I}$. Note that the maps $(F\circ \phi_x^{-1})_{x\in I}$ are the local representations of the Lip-$\gamma$ map $F$ on $\mathbb{R}^p$. Therefore, they are Lip-$\gamma$ with norms bounded uniformly from above.
	\end{proof}
	Next, we show that Lipschitz regularity is preserved on manifolds endowed with a Lipschitz structure induced by a Lipschitz homeomorphism.
	\begin{lemma}\label{lemma:LipMapBallManifold} Let $p\in \mathbb{N}^*$ and $\gamma\geq 1$. Let $V$ be a subset of a normed vector space and $f$ be a Banach-space valued Lip-$\gamma$ map defined on $V$. We assume that there exists a homeomorphism $F:\mathbb{R}^p\to V$ that is Lip-$\gamma$. Then $f$ is Lip-$\gamma$ on $V$ in the manifold sense in the $F$-induced Lipschitz structure.
	\end{lemma}
	\begin{proof}
	Consider the Lip-$\gamma$ atlas ${(U_x, \phi_x)}_{x\in I}$ on $\mathbb{R}^p$ constructed in Proposition \ref{Prop:EuclidLipManifoldEuclid}. Since $f\circ F$ is Lip-$\gamma$ on $\mathbb{R}^p$ (in the classical sense and hence also in the manifold sense), then there exists a constant $C$ such that, for all $x\in I$, $f\circ F \circ \phi_x^{-1}$ is Lip-$\gamma$ with a norm bounded from above by $C$. This concludes the proof (by the definitions of Lipschitz regularity on manifolds and the $F$-induced Lipschitz structure on $V$).
	\end{proof}
	\begin{Rem}\label{rem:LipMapUnitBall}
	By applying Lemma \ref{lemma:LipMapBallManifold} in the case of the unit ball, a map defined on the unit ball that is Lipschitz in the classical sense is also Lipschitz in the manifold sense when the unit ball is endowed with the Lipschitz structure of Corollary \ref{UnitBallLip}.
	\end{Rem}
	Like in the Euclidean setting (Whitney's extension theorems, e.g. \cite{Stein}), Lipschitz maps defined on compact subsets of Lipschitz manifolds and with values in a finite dimensional space can be extended.
	\begin{theo}\label{ExtendLipMapManifold}\cite{CLL2} Let $\gamma_0\geq\gamma>0$ such that $\gamma_0\geq 1$. Let $M$ be a Lip-$\gamma_0$ manifold and $K$ be a precompact subset of $M$. Let $f$ be a $V$-valued Lip-$\gamma$ map defined on $K$, where $V$ is a finite dimensional space. Then we can extend $f$ to a compactly supported Lip-$\gamma$ map on $M$ with a Lip-$\gamma$ norm depending only on $\|f\|_{\textrm{Lip}-\gamma}$, $M$, $K$, $\gamma_0$ and $\gamma$.
	\end{theo}
	\begin{Rem} In Theorem \ref{ExtendLipMapManifold}, the dependence on $M$ is through its Lipschitz structure (and the subsequent choice of a partition of unity) while the dependence on $K$ is through the set  of chart domains that intersect $K$.
	\end{Rem}
	\subsection{Lipschitz transports}\label{subsec:LipManifoldConn}
	The notion of transport will be key in translating the definition of solutions to RDEs from the classical to the manifold setting.
	\begin{Def}\label{Def:LipConnMan}
	Let $M$ and $N$ be two Lip-$\gamma$ manifolds of dimensions $n$ and $d$ respectively. We call transport from $M$ to $N$ a map $g$ defined on $M\times N$ such that, for all $x\in M$, $g(x,.) \in \mathcal{L}(T_xM,\mathcal{\tau}(N) )$. For such a map and for $(U,\phi)$ and $(V,\psi)$ charts in $M$ and $N$ respectively, we define $g_{(\phi,\psi)}$ by the following. For all $(x,y)\in \phi(U)\times \psi(V)$, $v \in \mathbb{R}^{n}$,
	\[
	g_{(\phi,\psi)}(x,y)(v)= \psi_*(\psi^{-1}(y))\left(g(\phi^{-1}(x),\psi^{-1}(y))(\phi^{-1})_*(x)(v)\right)
	\in \mathbb{R}^{d}.
	\]	
	We say that $g$ is Lip-$\gamma$ if there exists a constant $C$ such that for all charts $(U,\phi)$ and $(V,\psi)$ in $M$ and $N$ respectively, the map $g_{(\phi,\psi)}$  is Lip-$\gamma$ with norm less than $C$.
	\end{Def}
	\begin{example}
	If $E$ and $F$ are two vector spaces and $f: E\times F \to \mathcal{L}(E,F)$ is a map, then there exists a natural transport from $E$ to $F$ induced by $f$.
	\end{example}
	\begin{Rem}
	We will also sometimes consider the extension of Definition \ref{Def:LipConnMan} to transports that are not globally defined.
	\end{Rem}
	On the one hand, we will see later that RDEs (and by extension ODEs and flow equations) on a Lipschitz manifold admit a unique global solution when the vector fields define a Lipschitz transport map (with enough Lipschitz regularity in the case of RDEs). On the other hand, by Lemma \ref{lemma:LipMapBallManifold} and Remark \ref{rem:LipMapUnitBall}, vector fields on the unit ball that are Lipschitz in the classical sense are also Lipschitz in the manifold sense. However, it is an elementary exercise to construct smooth vector fields so that their flows starting from the unit ball leave the unit ball in a finite time. The next lemma reconciles these two arguments and sheds more light on Lipschitz geometry: Lipschitz vector fields do not necessarily define Lipschitz transports on the entire unit ball.
	\begin{lemma}\label{lemma:VFasLipConnMan}
	Let $n$ and $p$ be two positive integers. Let $f:B_{p}(0,1)\to \mathcal{L}(\mathbb{R}^{n}, \mathbb{R}^{p})$ be a Lip-$\gamma$ map. Define a transport $g$ from $\mathbb{R}^{n}$ to $B_{p}(0,1)$ by
	\[g(x,y):  v_x \in \mathbb{R}^{n} \longmapsto f(y)(v_x) \in \mathbb{R}^{p} ,\;\;\; \textrm{ for all } (x,y) \in \mathbb{R}^{n}\times B_{p}(0,1).\]
	Then :
	\begin{enumerate}
	\item $g$ is Lip-$\gamma$ on $\mathbb{R}^{n}\times B_{p}(0,\delta)$, for every $ \delta \in (0,1)$.
	\item if $f$ is compactly supported, then $g$ is Lip-$\gamma$ on $\mathbb{R}^{n}\times B_{p}(0,1)$.
	\end{enumerate}
	\end{lemma}
	\begin{proof}
	Consider the Lip-$(\gamma+1)$ atlases ${(U_x, \phi_x)}_{x\in I}$ and ${(\widetilde{U}_z, \varphi_z)}_{z\in J}$ on $\mathbb{R}^{n}$ and $\mathbb{R}^{p}$ respectively (as in Proposition \ref{Prop:EuclidLipManifoldEuclid}) and the Lipschitz bijection $F:\mathbb{R}^p \to B_p(0,1)$ of Lemma \ref{lemma:GlobalLipDiff}. For $z\in J$, we define $y_z=F(z)$, $V_{y_z}=F(\widetilde{U}_z)$ and $\psi_{y_z}= \varphi_z\circ F^{-1}$. Let $\delta \in (0,1)$. As the collection $\{B_{p}(z,2)\}_{z\in J}$ is locally finite and that $F^{-1}\left(B_{p}(0,\delta)\right)$ is precompact, then there exists a finite subset $J_\delta\subset J$ such that if $z\notin J_\delta$ then $B_{p}(z,2) \cap F^{-1}\left(B_{p}(0,\delta)\right)=\varnothing$ and consequently $V_{y_z} \cap B_{p}(0,\delta) = \varnothing$. Therefore, we restrict ourselves to the case $z\in J_\delta$. Let us finally note that as $\cup_{z\in J_\delta}F(\widetilde{U}_z)$ is precompact, then there exists $M_\delta \in (0,1)$ such that $\cup_{z\in J_\delta}F(\widetilde{U}_z) \subset B(0,M_\delta)$. Let $x\in I$ and $z\in J_\delta$. By Definition \ref{Def:LipConnMan}, we need to study the regularity of the map
	\[\begin{array}{rccl}
	g_{x,{y_z}}: & B_{n}(0,1) \times \psi_{y_z}\left( B_{p}(0,\delta)\cap V_{y_z} \right)&\longrightarrow&\mathcal{L}(\mathbb{R}^{n}, \mathbb{R}^{p}) \\
	&(u,v) &\longmapsto &	(\psi_{y_z})_*(\psi_{y_z}^{-1}(v))\circ f(\psi_{y_z}^{-1}(v))\circ (\phi_x^{-1})_*(u) \\
	\end{array}
	\]
	In anticipation of the second part of the proof, we will study instead the regularity of the (unrestricted) map
	\[g_{x,{y_z}}:  (u,v) \in B_{n}(0,1) \times B_{p}(0,1) \longmapsto 
	(\psi_{y_z})_*(\psi_{y_z}^{-1}(v))\circ f(\psi_{y_z}^{-1}(v))\circ (\phi_x^{-1})_*(u) 
	\in \mathcal{L}(\mathbb{R}^{n}, \mathbb{R}^{p}) .\]
	By Lemma \ref{lemma:LipMapBallManifold}, $f\circ \psi_{y_z}^{-1}$ is Lip-$\gamma$ with a constant that can be bounded from above independently of ${y_z}$. Therefore, by Proposition \ref{Compobifunc}, the map 
	\[(u,v) \in B_{n}(0,1) \times B_{p}(0,1) \longmapsto 
	f(\psi_{y_z}^{-1}(v))\circ (\phi_x^{-1})_*(u) 
	\in \mathcal{L}(\mathbb{R}^{n}, \mathbb{R}^{p})\]
	is Lip-$\gamma$ with a norm that is independent of both $x$ and ${y_z}$. By definition of the chart map $\psi_{y_z}$, one has
	\[\forall v \in B_{p}(0,1): \quad 
	(\psi_{y_z})_*(\psi_{y_z}^{-1}(v)) = \mathrm{d}(\varphi_z\circ F^{-1})\circ (F\circ \varphi_z^{-1}) (v).\]
	Since $F\circ \varphi_z^{-1}: B_{p}(0,1) \to B_{p}(0,M_\delta)$ is Lip-$\gamma$ and $\varphi_z\circ F^{-1}:  B_{p}(0,M_\delta) \to \mathbb{R}^p$ is Lip-$(\gamma+1)$ (with Lipschitz constant depending on $\delta$) then $v\mapsto (\psi_{y_z})_*(\psi_{y_z}^{-1}(v))$ is Lip-$\gamma$. Using again Proposition \ref{Compobifunc}, this implies that $g_{x,{y_z}}$ is indeed Lip-$\gamma$ (with Lipschitz norm bounded form above independently of $x$ and ${y_z}$). This concludes the first part of the proof.\\
	Assume now that $f$ is compactly supported. Let $\delta \in (0,1)$ such that $f$ is supported in $B_{p}(0,\delta)$. Let $x\in I$. By the above, $g_{x,y_z}$ is Lip-$\gamma$ for all $z\in J_\delta$ with a Lipschitz norm bounded from above by a constant that only depends on $\delta$. For $z\notin J_\delta$, $g_{x,y_z}$ is the null connection, therefore also Lip-$\gamma$. Hence $g$ is Lip-$\gamma$.
	\end{proof}
	Next, we show that the Lip-$\gamma'$ regularity of connections is stable by Lip-$\gamma$ diffeomorphisms for $\gamma'\leq \gamma -1$. To reduce the assumptions and make the construction more intuitive, we will restrict ourselves to the following specific situation that is relevant to our main result.
	\begin{Cor}\label{Cor:VFasLipConnMan2}
	Let $n$ and $d$ be two positive integers and $\gamma>1$. Let $E$ be a normed vector space and $V$ be a subset of $E$. Let $F: B_p(0,1)\to V$ be a Lip-$\gamma$ diffeomorphism and endow $V$ with the Lip-$\gamma$ structure induced by $F$. Let $f:V\to \mathcal{L}(\mathbb{R}^{n}, E)$ be a compactly supported Lip-$\gamma$ map such that for all $v\in \mathbb{R}^{n}$, $f(.)(v) \in \tau (V)$. Then the transport $g$ from $\mathbb{R}^{n}$ to $V$ induced by $f$ is Lip-$(\gamma-1)$.
	\end{Cor}
	\begin{proof}
	Consider the usual Lip-$(\gamma+1)$ atlas ${(U_x, \phi_x)}_{x\in I}$ and ${(\widetilde{U}_z, \varphi_z)}_{z\in J}$ on $\mathbb{R}^{n}$ and $B_p(0,1)$ respectively. For $z\in J$, we define $y_z=F(z)$, $V_{y_z}=F(\widetilde{U}_z)$ and $\psi_{y_z}= \varphi_z\circ F^{-1}$. Let $x\in I$ and $z\in J$. By Definition \ref{Def:LipConnMan}, we need to study the Lipschitz regularity of the map
	\[\begin{array}{rccl}
	g_{x,{y_z}}: & B_{n}(0,1) \times B_{p}(0,1)&\longrightarrow&\mathcal{L}(\mathbb{R}^{n}, \mathbb{R}^{p}) \\
	&(u,v) &\longmapsto &	(\psi_{y_z})_*(\psi_{y_z}^{-1}(v))\circ f(\psi_{y_z}^{-1}(v))\circ (\phi_x^{-1})_*(u) \\
	\end{array}
	\]
	Define the following compactly supported Lip-$(\gamma-1)$ on $B_p(0,1)$
	\[\widehat{f}:  z \in B_{p}(0,1) \longmapsto 
	\left( v\mapsto 
	(F^{-1})_*(F(z))(f\circ F)(z)(v)
	\right) \in \mathcal{L}(\mathbb{R}^{n}, \mathbb{R}^{p}),\]
	and denote by $\widehat{g}$ its associated transport from $\mathbb{R}^{n}$ to $B_p(0,1)$. Note then that $g_{x,{y_z}}=	\widehat{g}_{x,z}$. By Lemma \ref{lemma:VFasLipConnMan}, we conclude that $g$ is Lip-$(\gamma-1)$.
	\end{proof}
	\section{Local Lipschitz parametrisation of orbits of Lipschitz vector fields}\label{sec:LipOrbits}
	\subsection{Orbits and distributions}
	In this subsection, we will introduce the notion of orbits and their ``candidate'' tangent spaces. We will in most cases follow the terminology used in \cite{Sussmann} and will restrict the exposition to the minimum necessary to obtain our results. 
	\begin{Def}  Let $d\in\mathbb{N}^*$. Let $\mathcal{D}$ be a family of Lip-$1$ vector fields on $\mathbb{R}^d$. Let $n\in\mathbb{N}^*$. Let $f^1,\ldots,f^n\in\mathcal{D}$ and $t_1,\ldots,t_n\in \mathbb{R}$. The $\mathbb{R}^d$-valued map on $\mathbb{R}^d$ defined by $\tilde{f}^n_{t_n}\circ\cdots\circ \tilde{f}^1_{t_1}$ is called a $\mathcal{D}$-diffeomorphism.
	\end{Def}
	Orbits of vector fields will be a central notion in the rest of this paper.
	\begin{Def} Let $d\in\mathbb{N}^*$. Let $\mathcal{D}$ be a family of Lip-$1$ vector fields on $\mathbb{R}^d$ and $y \in \mathbb{R}^d$. We call $\mathcal{D}$-orbit at $y$ and denote by $\mathbb{L}_{y}(\mathcal{D})$ the set of images of $y$ by all $\mathcal{D}$-diffeomorphisms.
	\end{Def}
	In the sequel, we will use the following notations.
	\begin{Notation} Let $d\in\mathbb{N}^*$ and $\mathcal{D}$ be a family of Lip-$1$ vector fields on $\mathbb{R}^d$. Following similar notational conventions as before, for $n\in \mathbb{N}^*$ and $\xi=(f^1,\ldots,f^n)\in\mathcal{D}^n$,
	\begin{itemize}
		\item we denote by $\tilde{\xi}$ the map
	\[\tilde{\xi}:\mathbb{R}^n\times \mathbb{R}^d \longrightarrow \mathbb{R}^d \quad,\quad (t_1,\ldots,t_n,x)\longmapsto \tilde{f}^n_{t_n}\circ\cdots\circ \tilde{f}^1_{t_1}(x),\]
		\item for $y\in \mathbb{R}^d$, $\tilde{\xi}_y$ denotes the map $(t_1,\ldots,t_n)\mapsto \tilde{\xi}(t_1,\ldots,t_n,y)$,
		\item for $s=(s_1,\ldots,s_n)$ we denote by $\tilde{\xi}_{s}$ the $\mathcal{D}$-diffeomorphism $x\mapsto \tilde{\xi}(s_1,\ldots,s_n,x)$.
	\end{itemize}	 
	 More generally, for $p\in\mathbb{N}^*$, $n_1,\ldots,n_p\in\mathbb{N}^*$, $(\xi^1,\ldots,\xi^p)\in\mathcal{D}^{n_1}\times\cdots\times\mathcal{D}^{n_p}$ and, if $\eta$ denotes the ordered collection $(\xi^1,\ldots,\xi^p)$, then $\tilde\eta$ is the map given by the identity
	\[\forall (t_1,\ldots,t_p)\in\mathbb{R}^{n_1}\times\cdots\times\mathbb{R}^{n_p}, \forall x\in \mathbb{R}^d: \quad \tilde{\eta}(t_1,\ldots,t_p,x)=\tilde{\xi}^p_{t_p}\circ\cdots\circ\tilde{\xi}^1_{{t_1}}(x).\]
	For $y\in \mathbb{R}^d$, $\tilde{\eta}_y$ is the map $\tilde{\eta}(.,y)$ and for $s=({s_1},\ldots,{s_p})$, $\tilde{\eta}_{{s}}$ is the map $\tilde{\xi}^p_{{s_p}}\circ\cdots\circ\tilde{\xi}^1_{{s_1}}$. 
	\end{Notation}
	A key idea in Sussmann's work \cite{Sussmann} was to identify the ``candidate'' tangent space to orbits. This is best carried through the notion of distributions.
	\begin{Def} Let $M$ be a $\mathcal{C}^1$-manifold. A distribution $\Delta$ on $M$ is a mapping that associates to every point $p\in M$ a linear subspace of $T_{p}M$ which is denoted $\Delta(p)$.
	\end{Def}
	\begin{Def} Let $d\in\mathbb{N}^*$ and $\gamma\geq 1$. Let $\mathcal{D}$ be a family of Lip-$\gamma$ vector fields on $\mathbb{R}^d$. Let $\Delta$ and $\Gamma$ be two distributions on $\mathbb{R}^d$.
	\begin{itemize}
		\item We say that $\mathcal{D}$ spans $\Delta$ if, for every $p\in \mathbb{R}^d$, one has
		\[\Delta(p)=\textrm{span}\{f(p),\;f\in\mathcal{D}\}.\]
		\item If $\gamma>1$, we say that $\Delta$ is $\mathcal{D}$-invariant if, for every $\mathcal{D}$-diffeomorphism $g$ and $y\in \mathbb{R}^d$, $g_*(y)$ induces a diffeomorphism from $\Delta(y)$ onto $\Delta(g(y))$.
		\item We say that $\Delta$ contains $\Gamma$ and write $\Gamma\subseteq\Delta$ if, for every $y\in \mathbb{R}^d$, $\Gamma(y)\subseteq\Delta(y)$.
	\end{itemize}
	\end{Def}
	One is then able to identify the candidate tangent space to a $\mathcal{D}$-orbit as the collection of tangent vectors spanned by the vector fields of $\mathcal{D}$ and the transport of these via $\mathcal{D}$-diffeomorphisms.
	\begin{theo}{\cite{Sussmann}} Let $d\in\mathbb{N}^*$ and $\gamma\geq 1$. Let $\mathcal{D}$ be a family of Lip-$\gamma$ vector fields on $\mathbb{R}^d$.
	\begin{enumerate}
		\item There exists a unique natural distribution spanned by $\mathcal{D}$. This will be denoted by $\mathcal{L}(\mathcal{D})$.
		\item Assume that $\gamma>1$. Then there exists a unique smallest distribution $P_{\mathcal{D}}$ containing  $\mathcal{L}(\mathcal{D})$ that is $\mathcal{D}$-invariant. More explicitly, for every $y\in \mathbb{R}^d$, $P_{\mathcal{D}}(y)$ is given by
	\[P_{\mathcal{D}}(y):=
	\left\{
	g_*(v) \; : \; 
	g \;\; \mathcal{D}-\textrm{diffeomorphism} , v \in \mathcal{L}(\mathcal{D})(g^{-1}(y))
	\right\}.\]
	\end{enumerate}
	\end{theo}
	As a direct consequence of the definition of invariance, one has the following.
	\begin{Cor}{\cite{Sussmann}} Let $d\in\mathbb{N}^*$ and $\gamma> 1$. Let $\mathcal{D}$ be a family of Lip-$\gamma$ vector fields on $\mathbb{R}^d$. Then the dimension of $P_{\mathcal{D}}$ is constant on $\mathcal{D}$-orbits, i.e. if $x$ and $y$ belong to the same $\mathcal{D}$-orbit then $\mathrm{dim}\left( P_{\mathcal{D}}(x)\right)=\mathrm{dim}\left( P_{\mathcal{D}}(y)\right)$.
	\end{Cor}
	While $P_{\mathcal{D}}$ is, by construction, rich enough to contain all the tangent vectors generated by integral curves of the vector fields in $\mathcal{D}$ and their compositions, it can still be precisely generated by these. This is the content of the theorem below, which is the summary of the three main lemmas in \cite{Sussmann}. The only difference with the original statements is the Lipschitz regularity condition instead of the $\mathcal{C}^\infty$ assumption, which does not impact the validity of the original proofs.
	\begin{theo}{\cite{Sussmann}}\label{theo:GeneratePD} Let $d\in\mathbb{N}^*$ and $\gamma> 1$. Let $\mathcal{D}$ be a family of Lip-$\gamma$ vector fields on $\mathbb{R}^d$. 
	\begin{enumerate}
		\item For $n\in\mathbb{N}^*$, $\xi\in\mathcal{D}^n$, $x,y\in \mathbb{R}^d$ and $T\in \mathbb{R}^n$ such that $\tilde{\xi}(T, x)=y$, one has
		\[(\tilde{\xi}_{x})_*(T)(\mathbb{R}^n)\subseteq P_{\mathcal{D}}(y).\]
		\item Conversely, for all $y\in \mathbb{R}^d$, there exist $x\in \mathbb{R}^d$, $n\in\mathbb{N}^*$, $\xi\in\mathcal{D}^n$ and $T\in \mathbb{R}^n$ such that
	\[\tilde{\xi}(T,x)=y \textrm{ and } (\tilde{\xi}_{x})_*(T)(\mathbb{R}^n)= P_{\mathcal{D}}(y).\]
	\end{enumerate}
	\end{theo}
	\subsection{A local Lipschitz structure on orbits}
	We obtain now one of the main results in this paper by endowing neighbourhoods of points in their orbits with Lipschitz structures.
	\begin{theo}\label{LipCharts} Let $d\in\mathbb{N}^*$ and $\gamma>1$. Let $\mathcal{D}$ be a family of Lip-$\gamma$ vector fields on $\mathbb{R}^d$. Let $y\in \mathbb{R}^d$. Then there exist a subset $V_y$ of $\mathbb{R}^d$ containing $y$ and contained in $\mathbb{L}_{y}(\mathcal{D})$ and a Lip-$\gamma$ diffeomorphism $\Psi_y: V_y \to B_p(0,1)$ centred at $y$, where $p=\textrm{dim}( P_{\mathcal{D}}(y))$. Moreover, $V_y$ has a Lip-$\gamma$ structure induced by $\Psi_y^{-1}$ with its tangent space given precisely by the distribution $P_{\mathcal{D}}$.
	\end{theo}
	\begin{proof}
	By Theorem \ref{theo:GeneratePD}, let $n\in\mathbb{N}^*$,  $\xi\in\mathcal{D}^n$, $T\in\mathbb{R}^n$ and $x\in \mathbb{R}^d$ such that $\tilde{\xi}(T,x)=y$ and $(\tilde{\xi}_{x})_*(T)(\mathbb{R}^n)= P_{\mathcal{D}}(y)$. By the same theorem, for every $t\in\mathbb{R}^n$,
	\[(\tilde{\xi}_{x})_*({t})(\mathbb{R}^n)\subseteq P_{\mathcal{D}}(\tilde{\xi}_{x}({t})).\]
	Since $\mathrm{dim}( P_{\mathcal{D}}(\tilde{\xi}_{x}({t})))= p$, then $\mathrm{rank}(\tilde{\xi}_{x}({t}))_*\leq p$. Therefore the map $(\tilde{\xi}_{x})_*$ is of maximal rank $p$ at $T$. Moreover, by Theorems \ref{theo:LipRegFlow} and \ref{thm:CompoLip}, $\tilde{\xi}_{x}$ defines a Lip-$\gamma$ map on every open bounded set of $\mathbb{R}^n$. By Theorem \ref{ConstantRankTheo}, there exist then an open subset $U$ of $\mathbb{R}^n$ containing $T$, an open subset $V$ of $\mathbb{R}^d$ containing $y$, Lip-$\gamma$ diffeomorphisms $f: U\to B_n(0,1)$ and $g:V\to B_d(0,1)$ (centred at $T$ and $y$ respectively) such that the following diagram commutes
	\[\begin{array}{ccc}
	U (\subseteq \mathbb{R}^n)&\overset{\tilde{\xi}_{x}}\longrightarrow&V(\subseteq \mathbb{R}^d)\\
	f\Big\downarrow&&\Big\downarrow g\\
	B_n(0,1)&\overset{\pi_{n,p,d}}\longrightarrow& B_d(0,1)\\
	\end{array}\]
	In particular, for all $t\in U$, $({\tilde{\xi}_{x}})_{*}(t)$ is of rank $p$.	Define $V_y=\tilde{\xi}_{x}(U)$ and the map $\Psi_y:V_y \to B_p(0,1)$ as the composition $\Psi_y:=\pi_{d,p,p}\circ g$. Then $\Psi_y$ is invertible and $\Psi_y^{-1}=g^{-1}\circ \pi_{p,p,d}$. Hence, by Proposition \ref{CompoLinFunc}, $\Psi_y$ is a Lip-$\gamma$ diffeomorphism. This induces a natural $p$-dimensional Lip-$\gamma$ structure on $V_y$ by Proposition \ref{Prop:TransportLipStructure} and Corollary \ref{UnitBallLip}. For $t\in U$, decomposing the invertible map $\Psi_y^{-1}$ as $\Psi_y^{-1}= \tilde{\xi}_{x} \circ f^{-1}\circ \pi_{p,p,n}$, one gets
	\[T_{\tilde{\xi}(t,x)}V_y\subseteq \tilde{\xi}_{x}(t)_*(\mathbb{R}^n) \subseteq P_{\mathcal{D}}(\tilde{\xi}(t,x)).\]
	As $\tilde{\xi}(t,x)\in\mathbb{L}_{y}(\mathcal{D})$, then $P_{\mathcal{D}}(\tilde{\xi}(t,x))$ is of the same dimension as $P_{\mathcal{D}}(y)$, i.e. $p$, the same as $T_{\tilde{\xi}(t,x)}V_y$. Therefore, we conclude that $T_{\tilde{\xi}(t,x)}V_y= P_{\mathcal{D}}(\tilde{\xi}(t,x))$.
	\end{proof}	
	\begin{Rem} Using Theorem \ref{LipCharts}, one may build a locally Lip-$\gamma$ structure on $\mathbb{L}_{y}(\mathcal{D})$ as defined in \cite{BL}. However, as we didn't show the existence of solutions to rough differential equations in such setting, we will restrict ourselves to local arguments using the above propo\-sition. 
	\end{Rem}
	\section{Rough Calculus on manifolds}\label{sec:RoughOnManifolds}
	\subsection{Rough paths on manifolds}
	In order to carry out our main argument, we need to generalise the notion of rough paths to manifolds.  As rough calculus on manifolds is less known among rough paths specialists, and in view of the multiple manifold approaches available in the literature, we will briefly recall in this section the main definitions and results that we will subsequently use, most of which are due to \cite{CLL2}.\\
	
	We fix a notation for the integral of a rough path, when endowed with a starting point, along a one-form.
	\begin{Notation} Let $\gamma>p\geq 1$ and $T\geq 0$. Let $E$ and $F$ be two Banach spaces. Let $x\in E$ and $\mathbf{X}$ be a geometric $p$-rough path on $E$ defined over the interval $[0,T]$. Let $\alpha:E\to \mathcal{L}(E,F)$ be a Lip-$(\gamma-1)$ one-form. We denote by $\int\alpha(x,\mathbf{X})\mathrm{d}\mathbf{X}$ the integral of the rough path $\mathbf{X}$ along $\alpha$ when endowed with $x$ as a starting point, i.e. the trace of $\mathbf{X}$ is understood to be given by the path $t\mapsto x+\mathbf{X}^1_{0,t}$.
	\end{Notation}
	
	In the absence of the linear structure, rough paths on Lipschitz manifolds can be defined using a functional approach.
	\begin{Def}\cite{CLL2}\label{Def:RPManCLLy} Let $\gamma_0, p\in \mathbb{R}$ such that $\gamma_0>p\geq1$ and $T\geq0$. Let $M$ be a $\textrm{Lip}-\gamma_0$ manifold and $x\in M$. $\mathbb{X}$ is a (geometric) $p$-rough path over $[0,T]$ on $M$ starting at $x$ if, for every $\gamma\in \mathbb{R}$ such that $\gamma_0\geq\gamma>p$ and every Banach space $E$, the following conditions are satisfied:\begin{enumerate}
	\item $\mathbb{X}$ maps compactly supported $\textrm{Lip}-(\gamma-1)$ $E$-valued one-forms on $M$ to $E$-valued geometric $p$-rough paths (in the classical sense).
	\item There exists a control $\omega$ such that for every compactly supported $E$-valued $\textrm{Lip}-(\gamma-1)$ one-form $\alpha$ on $M$, $\mathbb{X}(\alpha)$ is controlled by $\|\alpha\|_{\textrm{Lip}-(\gamma-1)}\omega$, i.e.
	\[\forall (s,t)\in \Delta_{[0,T]}, \forall i\in [\![1,[p]]\!]:\quad \|\mathbb{X}(\alpha)_{(s,t)}^i\|\leq\frac{(\|\alpha\|_{\textrm{Lip}-(\gamma-1)}\omega(s,t))^{i/p}}{\beta_p (i/p)!}.\]
	\item (Chain rule) For every Banach space $F$ and every compactly supported $\textrm{Lip}-\gamma$ map $\psi: M\to F$ and every $E$-valued $\textrm{Lip}-(\gamma-1)$ one-form $\alpha$ on $F$ we have
	\[\mathbb{X}(\psi^*\alpha)=\int{\alpha (\psi(x),\mathbb{X}(\psi_*)) \mathrm{d}\mathbb{X}(\psi_*)}.\]
	\end{enumerate}
	\end{Def}
	Naturally, the classical and manifold notions of rough paths on finite-dimensional vector spaces (endowed with their natural Lipschitz structures) are the same. More explicitly, given a geometric $p$-rough path $\mathbf{X}$ with a starting point $x$ in $\mathbb{R}^d$, one can define a $p$-rough path $\mathbb{X}$ on $\mathbb{R}^d$ in the manifold sense in the following way. For every compactly supported Banach space-valued $\textrm{Lip}-(\gamma-1)$ one-form $\alpha$ on $\mathbb{R}^d$
	\[\mathbb{X}(\alpha):=\int\alpha(x,\mathbf{X})\mathrm{d}\mathbf{X}.\]
	Conversely, the functional $\mathbb{X}$ completely determines the classical rough path $\mathbf{X}$. The proof relies on all three axioms of Definition \ref{Def:RPManCLLy} and can be found in \cite{CLL2}.\\
	
	Next, we adapt the notion of support of rough paths.
	\begin{Def}\cite{CLL2} Let $\gamma_0> p\geq 1$ and $M$ be a $\textrm{Lip}-\gamma_0$ manifold. Let $\mathbb{X}$ be a $p$-rough path on $M$ with starting point $x$. 
	\begin{enumerate}
	\item Given an open subset $U$ of $M$, we say that $\mathbb{X}$ misses $U$ if, for every $\gamma \in (p, \gamma_0]$ and every Banach space-valued compactly supported Lip-$(\gamma-1)$ one-form $\alpha$ on $M$ with support in $U$, we have $\mathbb{X}(\alpha)=0$.
	\item We define the support of the rough path $\mathbb{X}$ as follows
	\[\mathrm{supp}(\mathbb{X}):=\{x\}\bigcup \left(M-\bigcup_{\mathbb{X} \textrm{ misses } U}U\right).\]
	\end{enumerate}
	\end{Def}
	Once again, the classical and manifold notions of support of a rough path coincide. In this case, the manifold support of a classical rough path (with a starting point) viewed as a manifold rough path is given by the image of its trace.\\
	More generally, this manifold notion of support generalises several key properties of the classical supports of rough paths. For instance, the support of a rough path (on a manifold) is a compact set \cite{CLL2}. Moreover, the integrals of a rough path along two one-forms that agree on an open set containing its support are the same.
	\begin{Prop}\cite{CLL2}\label{RPIntegralDependenceSupport} Let $\gamma_0\geq \gamma> p\geq 1$. Let $M$ be a Lip-$\gamma_0$ manifold, $U$ be an open subset of $M$ and $E$ be a Banach space. Let $\alpha$ and $\beta$ be two $E$-valued compactly supported Lip-$(\gamma-1)$ one-forms such that $\alpha_{|U}=\beta_{|U}$. Let $\mathbb{X}$ be a $p$-rough path in $M$ such that $\mathrm{supp}(\mathbb{X})\subset U$, then $\mathbb{X}(\alpha)= \mathbb{X}(\beta)$.
	\end{Prop}
	The rigid nature of the axioms in Definition \ref{Def:LipManCLL2} has the benefit of ensuring that  the restrictions of a rough path to shorter intervals are supported in the domains of the Lipschitz charts; the lengths of these intervals depending only on the control of the rough path.
	\begin{theo}\cite{CLL2}\label{SingleChartRP} Let $\gamma_0> p\geq 1$ and $\omega$ be a control. Let $M$ be a $\textrm{Lip}-\gamma_0$ manifold. Then there exists $t_0>0$ such that, for every $p$-rough path $\mathbb{X}$ on $M$ controlled by $\omega$, there exists a chart domain $U$ such that $\mathrm{supp}(\mathbb{X}_{|[0,t_0]})\subset U$.
	\end{theo}	
	\begin{Rem} With the notations of Definition \ref{Def:LipManCLL2} and Theorem \ref{SingleChartRP}, the chart $(\phi,U)$ can be chosen to be any such chart such that $B(\phi(x), \delta)\subset B(0,1)$, $x$ being the starting point of $\mathbb{X}$. The time $t_0$ depends only the constants involved in the definition of the Lipschitz atlas on $M$ and the control $\omega$.
	\end{Rem}
	\subsection{RDEs on manifolds}
	In this subsection, we will recall the most important results of the theory of rough differential equations on manifolds as developed in \cite{CLL2}.	
	\begin{Notation}
	Let $M$ and $N$ be two smooth manifolds and $g$ be a transport map from $M$ to $N$.
	If $\alpha$ is a Banach space-valued one form on $M\times N$, we denote by $\alpha^g$ the one-form defined on  $M\times N$ as follows. For all $z\in M\times N$ and $v=(v_x,v_y) \in T_z(M\times N)$,
	\[	\alpha^g(z)(v):=
	 \alpha(z)(v_x, g(z)(v_x)). 
	\]
	\end{Notation}
	We can now introduce a functional definition of solutions to RDEs on manifolds.
	\begin{Def}\label{Def:RDE-Man} \cite{CLL2} Let $\gamma> p\geq 1$. Let $M$ and $N$ be two Lip-$\gamma$ manifolds, $x_0 \in M$ and $y_0\in N$. Let $\mathbb{X}$ be a $p$-rough path in $M$ with starting point $x_0$. Let $g$ be a transport map from $M$ to $N$. We say that a rough path $\mathbb{Y}$ in $N$ with starting point $y_0$ is a solution to the RDE
	\begin{equation}\label{Eq:MainRDEMan}
	\left\{\begin{array}{rcll}
	\mathrm{d}\mathbb{Y}_t &=& g(\mathbb{X}_t,\mathbb{Y}_t)\mathrm{d}\mathbb{X}_t,& \quad \forall t\leq T,\\
	\mathbb{Y}_0&=& y_0\\
	\end{array}
	\right.
	\end{equation}
	if there exists a $p$-rough path $\mathbb{Z}$ in $M\times N$ with stating point $(x_0,y_0)$ such that $\pi_M(\mathbb{Z})=\mathbb{X}$, $\pi_N(\mathbb{Z})=\mathbb{Y}$ and for all compactly supported Banach-space valued Lip-$(\gamma-1)$ one-forms $\alpha$ on $M\times N$, one has $\mathbb{Z} (\alpha)=\mathbb{Z} (\alpha^{g})$.
	\end{Def}
	\begin{theo}\label{UniqueSolRDEMan}\cite{CLL2} In the setting of Definition \ref{Def:RDE-Man}, if $g$ is Lip-$\gamma$ (in the sense of Definition \ref{Def:LipConnMan}), then there exists a unique solution to the RDE (\ref{Eq:MainRDEMan}).
	\end{theo}
	\section{The accessibility property of geometric RDEs}\label{sec:MainThm}
	\subsection{Local solutions on orbits}
	We now show that one may make sense of and locally solve RDEs on an orbit of Lipschitz vector fields.
	\begin{Prop}\label{ExistSolSubm} Let $p\geq 1$ and $\gamma>p+1$. Let $n, d\in\mathbb{N}^*$ and $T>0$. Let $\mathcal{D}=\{f^1,\ldots, f^n\}$ be a family of Lip-$\gamma$ vector fields on $\mathbb{R}^d$ and define the map $f:\mathbb{R}^d\to \mathcal{L}(\mathbb{R}^n,\mathbb{R}^d)$ by
	\[ \forall y\in \mathbb{R}^d, \forall x=(x_1,\ldots,x_n) \in \mathbb{R}^n: \quad
	f(y)(x)=\sum_{i=1}^n x_if^i(y).
	\]
	Let $y\in \mathbb{R}^d$. Let $\mathbb{X}$ be a geometric $p$-rough path on $\mathbb{R}^n$. Then there exists $T_0\in (0,T]$ such that the rough differential equation
	\begin{equation}\label{eq:RDESubmanifold}
	\left\{\begin{array}{l}
	\mathrm{d}\mathbb{Y}_t=f(\mathbb{Y}_t)\mathrm{d}\mathbb{X}_t\\
	\mathbb{Y}_0=y\\
	\end{array}
	\right.
	\end{equation}
	admits a solution $\mathbb{Y}$ as a $p$-rough path in $V_y$ (defined in Theorem \ref{LipCharts}) defined over $[0,T_0]$.
	\end{Prop}
	\begin{proof} First note that since for all $v\in \mathbb{R}^n$ and $u\in V_y$, 
	\[f(u)(v)\in \mathcal{L}\left(\mathcal{D}\right)(u)\subseteq P_{\mathcal{D}}(u),\]
	the differential equation $(\ref{eq:RDESubmanifold})$ is well defined on $V_y$. Let $p=\mathrm{dim}P_{\mathcal{D}}(y)$. By Lemma \ref{lemma:LipMapBallManifold}, $f$ is Lipschitz on $V_y$ in the manifold sense. Let $\widehat{f}$ be a compactly supported Lip-$\gamma$ map with support in $\Psi_y^{-1}\left(B_p\left(0,\frac{1+\sqrt{2}}{2\sqrt{2}} \right)\right)$
	and such that it agrees with $f$ on 
	$\Psi_y^{-1}\left(\overline{B_p\left(0,1/\sqrt{2}\right)}\right)$\footnote{The choice of the value $1/\sqrt{2}$ is merely due to the fact that $B_p\left(0,1/\sqrt{2}\right)$ is a Lipschitz chart domain in $B_p\left(0,1\right)$ in the Lipschitz structure endowed by $F$ in Lemma \ref{lemma:GlobalLipDiff}.}. Then, by Corollary \ref{Cor:VFasLipConnMan2}, the corresponding transport $\widehat{g}$ is Lip-$(\gamma-1)$. By Theorem \ref{UniqueSolRDEMan} (and since $\gamma-1>p$), let $\widehat{\mathbb{Z}}=(\mathbb{X} ,\widehat{\mathbb{Y}})$ be a solution to the rough differential equation defined on $\mathbb{R}^n \times V_y$ by
	\[
	\left\{\begin{array}{l}
	\mathrm{d}\widehat{\mathbb{Y}}_t=\widehat{f}(\widehat{\mathbb{Y}}_t)\mathrm{d}\mathbb{X}_t\\
	\widehat{\mathbb{Y}}_0=y\\
	\end{array}
	\right.
	\]
	By Theorem \ref{SingleChartRP} (and the subsequent remark), let $T_0\in (0,T]$ such that $\mathbb{Z}:=\widehat{\mathbb{Z}}_{[0,T_0]}$ is supported in the chart domain $B_n(0,1) \times \Psi_y^{-1}\left(B_p\left(0,1/\sqrt{2}\right)\right)$ (where its starting point also lives). Let $\alpha$ be a compactly supported Lip-$(\gamma_0-1)$ Banach space valued one form on $\mathbb{R}^n \times V_y$, with $\gamma_0 \in (p,\gamma]$. By the definition of the solution to an RDE on a manifold, one has $\mathbb{Z}(\alpha)=\mathbb{Z}(\alpha^{\widehat{f}})$ (where $\widehat{f}$ is identified with its induced transport map). For $(x,u)\in \mathbb{R}^n \times V_y$ and $(v_x, v_u)\in \mathbb{R}^n \times T_uV_y$, one has by definition
	\[
	\alpha^{\widehat{f}}(x,u)(v_x,v_u)
	= \alpha(x,u)(v_x,\widehat{f}(u)(v_x)).
	\]
	From the above expression, one sees that $\alpha^{\widehat{f}}$ and $\alpha^{f}$ agree in the chart domain $B_n(0,1) \times \Psi_y^{-1}\left(B_p\left(0,1/\sqrt{2}\right)\right)$ which contains the support of $\mathbb{Z}$. Therefore, by Proposition \ref{RPIntegralDependenceSupport}, $\mathbb{Z}(\alpha^{\widehat{f}})	=\mathbb{Z}(\alpha^{f})$ and consequently
	\[\mathbb{Z}(\alpha)
	=\mathbb{Z}(\alpha^{\widehat{f}})
	=\mathbb{Z}(\alpha^{f}).
	\]
	Therefore $\mathbb{Z}$ is a solution to $(\ref{eq:RDESubmanifold})$ on $[0,T_0]$.\end{proof}
	\begin{Rem}
	The careful reader may have noticed the stronger than usual condition on the regularity of the vector fields $\gamma>p+1$ (instead of $\gamma>p$) in the statement of Proposition \ref{ExistSolSubm}. Let us say a few words about this:
	\begin{itemize}
		\item While one can theoretically assume first order objects on manifolds (such as vector fields, one-forms and transport maps) to be of any regularity, it is not geometrically sensible to assume these objects to be of at least the same regularity as the manifold itself: for example, one speaks only of $\mathcal{C}^{n-1}$ vector fields on $\mathcal{C}^{n}$ manifolds. There are two interlinked reasons for this. First, whenever two chart $(U_1,\varphi_1)$ and $(U_2,\varphi_2)$ have intersecting domains, one usually would want the regularity of such objects on the domain $U_1\cap U_2$ when pushed forward by $\varphi_2$ to be deduced from its regularity when pushed forward by $\varphi_1$. In this case, one encounters the pushforward of the transition map $\varphi_1 \circ \varphi_2^{-1}$, which is one degree less regular than the regularity of the manifold. Second, and in most practical and non-trivial examples, one cannot prove such strong regularity if derivatives are involved, e.g. Corollary \ref{Cor:VFasLipConnMan2}. Thus, Theorem \ref{UniqueSolRDEMan} holds mainly theoretically if we only assume $\gamma>p$.
		\item It will become clear later that, on the manifold $V_y$, one only needs an existence result for solutions to RDEs. It will be enough then to identify these manifold solutions with the classical ones on $\mathbb{R}^d$. The proof of Theorem \ref{UniqueSolRDEMan} is entirely based on the similar result (called the universal limit theorem) holding on Banach spaces (e.g. \cite{Lyons,CLL}): the solution is first constructed on the unit balls before being pushed forward onto the manifold. This existence and uniqueness result strongly relies on the assumption of the transports being of regularity $\gamma$ with $\gamma>p$. To relax this regularity assumption to $\gamma-1$, one may attempt to build on the very known existence (only) result by Davie \cite{Davie} then Friz and Victoir \cite{FV, FV2}. However, we will restrict ourselves here to the framework of the theory given in \cite{CLL2} as our aim is to illustrate the geometric argument used to answer the accessibility question.
	\end{itemize}
	\end{Rem}
	\subsection{Identification with classical solutions}
	The next two lemmas allow us to ``build our way back'' from local solutions to RDEs on the orbits to classical solutions on the Euclidean space.
	\begin{lemma}\label{OneFormSubman} Let $n, d\in\mathbb{N}^*$ and $\gamma>1$. Let $\mathcal{D}$ be a family of Lip-$\gamma$ vector fields on $\mathbb{R}^d$. Let $y\in \mathbb{R}^d$. Then for every Banach space-valued Lip-$(\gamma-1)$ one-form $\alpha$ on $\mathbb{R}^n\times \mathbb{R}^d$, $\alpha_{|\mathbb{R}^n\times V_y}$ is a Lip-$(\gamma-1)$ one-form on $\mathbb{R}^n\times V_y$, and there exists a constant $c_{\gamma, y}$ depending only on $\gamma$ and $\Psi_y$ such that
	\[\|\alpha_{|\mathbb{R}^n\times V_y}\|_{\textrm{Lip-}(\gamma-1)}\leq c_{\gamma, y} \|\alpha\|_{\textrm{Lip-}(\gamma-1)}.\]
	\end{lemma}
	\begin{proof} We denote $p=\textrm{dim}( P_{\mathcal{D}}(y))$. Let ${(U_x, \phi_x)}_{x\in I}$ and ${(\widetilde{U}_z, \varphi_z)}_{z\in J}$ be the Lip-$\gamma$ atlases  on $\mathbb{R}^{n}$ and $V_y$ respectively. For $(x,z)\in I\times J$, we denote
	\[
	\xi_{(x,z)}: (u,v)\in \mathbb{R}^{n} \times V
	\longmapsto (\phi_x(u),\varphi_z(v)) \in B_n(0,1)\times B_p(0,1).
	\]
	To prove the claim, we need to show that the one-forms $((\xi^{-1}_{(x,z)})_*\alpha )_{(x,z)\in I\times J}$ are Lip-$\gamma$ with uniformly bounded norms. By the definition of $\phi_x$ and Lemma \ref{lemma:InvChartLip}, the maps $(\xi^{-1}_{(x,z)})_{(x,z)\in I\times J}$ are almost Lip-$\gamma$ with uniformly bounded constants. Hence the result by Lemma \ref{lemma:PullBackOneForm}.
	\end{proof}	
	\begin{lemma}\label{ExtendRPM} Let $\gamma, p \in \mathbb{R}$ such that $\gamma>p\geq1$. Let $n, d\in\mathbb{N}^*$. Let $\mathcal{D}$ be a family of Lip-$\gamma$ vector fields on $\mathbb{R}^d$. Let $y\in \mathbb{R}^d$. Then every $p$-rough path $\mathbb{Z}$ in $\mathbb{R}^n\times V_y$ with starting point $z$ defines a $p$-rough path $\widehat{\mathbb{Z}}$ in $\mathbb{R}^n\times \mathbb{R}^d$ with the same starting point and such that, for every compactly supported Banach space-valued Lip-$(\gamma-1)$ one-form $\alpha$ on $\mathbb{R}^n\times \mathbb{R}^d$
	 \[
	 \widehat{\mathbb{Z}}(\alpha)=\mathbb{Z}(\alpha_{|\mathbb{R}^n\times V_y})
	 .\]
	\end{lemma}
	\begin{proof} Let $\mathbb{Z}$ be a $p$-rough path in $\mathbb{R}^n\times V_y$ with starting point $z$ and with $p$-variation controlled by a control function $\omega$. We define the functional $\widehat{\mathbb{Z}}$ acting on any compactly supported Banach space-valued Lip-$(\gamma_0-1)$ one-form (with $\gamma\geq \gamma_0>p$) $\alpha$ on $\mathbb{R}^n\times \mathbb{R}^d$ by $\widehat{\mathbb{Z}}(\alpha)=\mathbb{Z}(\alpha_{|\mathbb{R}^n\times V_y})$. Let $\alpha$ be such a one-form. Then by the definition of a rough path on a manifold and Lemma \ref{OneFormSubman} , $\widehat{\mathbb{Z}}(\alpha)$ is a $p$-rough path. Moreover, for every $0\leq s\leq t\leq T$ and $i\in [\![1,p]\!]$, one has (using the bound obtained in Lemma \ref{OneFormSubman})
	 \[
	 \left\|\widehat{\mathbb{Z}}(\alpha)^i_{s,t}
	 \right\|
	 =\left\|\mathbb{Z}(\alpha_{|\mathbb{R}^n\times V_y})^i_{s,t}
	 \right\|
	 \leq \frac{\left(c_{\gamma, y} \|\alpha\|_{\textrm{Lip-}(\gamma-1)}\omega(s,t)\right)^{i/p}}{\beta_p \left(\frac{i}{p} \right)!}
	 .\]
	 Note that $c_{\gamma, y} \omega$ is also a control function. Finally, given a compactly supported Lip-$\gamma_0$ map $g$ defined on $\mathbb{R}^n\times \mathbb{R}^d$ with values in a Banach space $W$ and a Banach space-valued Lip-$(\gamma_0-1)$ one-form $\alpha$ on $W$, it suffices to note that
	 \[(g^*\alpha)_{|\mathbb{R}^n\times V_y}=(g_{|\mathbb{R}^n\times V_y})^*\alpha\]
	 to prove the chain rule property for $\widehat{\mathbb{Z}}$ using that of $\mathbb{Z}$. Hence $\widehat{\mathbb{Z}}$ is indeed a $p$-rough path on $\mathbb{R}^n\times \mathbb{R}^d$.
	\end{proof}
	We show now that, for a short time, the (traces of the) solutions to a classical RDE remain in the same orbit (associated to its vector fields).
	\begin{lemma}\label{lemma:transferSolSubmMan} Let $p\geq 1$ and $\gamma>p+1$. Let $n, d\in\mathbb{N}^*$. Let $\mathcal{D}=\{f^1,\ldots, f^n\}$ be a family of Lip-$\gamma$ vector fields on $\mathbb{R}^d$ and define the map $f:\mathbb{R}^d\to \mathcal{L}(\mathbb{R}^n,\mathbb{R}^d)$ by
	\[ \forall y\in \mathbb{R}^d, \forall x=(x_1,\ldots,x_n) \in \mathbb{R}^n: \quad
	f(y)(x)=\sum_{i=1}^n x_if^i(y).
	\] 
	Let $y\in \mathbb{R}^d$. Let $\mathbf{X}$ be a geometric $p$-rough path on $\mathbb{R}^n$. Consider the following classical RDE on $ \mathbb{R}^d$
	\begin{equation}\label{eq:RDESubmanifold_2}
	\left\{\begin{array}{l}
	\mathrm{d}\mathbf{Y}_t=f(\mathbf{Y}_t)\mathrm{d}\mathbf{X}_t\\
	\mathbf{Y}_0=y\\
	\end{array}
	\right.
	\end{equation}
	Then there exists $T_0>0$ such that the trace of the solution to (\ref{eq:RDESubmanifold_2})   on the time interval $[0,T_0]$ lives in $V_y$, i.e. for all $t\in [0,T_0]$ one has $y+\mathbf{Y}^1_{0,t} \in V_y$.
	\end{lemma}
	\begin{proof} Let $\mathbb{Z}$ be a solution to the RDE (\ref{eq:RDESubmanifold_2}) considered in the manifold sense on $V_y$ defined over a non-trivial interval $[0,T_0]$ (Proposition \ref{ExistSolSubm}). By Lemma \ref{ExtendRPM}, $\mathbb{Z}$ defines a rough path $\widehat{\mathbb{Z}}$ on $\mathbb{R}^n\times \mathbb{R}^d$. Let $\alpha$ be a Lip-$(\gamma_0-1)$ compactly supported Banach-space-valued one form on $\mathbb{R}^n\times \mathbb{R}^d$, with $\gamma_0\in(p,\gamma]$. By the definition of solutions to RDEs on manifolds and the very definition of $\widehat{\mathbb{Z}}$, one has the following
	\[\widehat{\mathbb{Z}}(\alpha)
	=\mathbb{Z}(\alpha_{|\mathbb{R}^n\times V_y})
	=\mathbb{Z}\left(\left(\alpha_{|\mathbb{R}^n\times V_y}\right)^{f}\right)
	=\mathbb{Z}\left(\left(\alpha^{f}\right)_{|\mathbb{R}^n\times V_y}\right)
	=\widehat{\mathbb{Z}}\left(\alpha^{f}\right).\]
	Therefore $\widehat{\mathbb{Z}}$ is a solution to the RDE (\ref{eq:RDESubmanifold_2}) on $\mathbb{R}^d$ in the manifold sense; and by the consistency of definitions, also in the classical sense. By the uniqueness of such solution, we have then $\mathbf{Y}=\pi_{\mathbb{R}^d}(\widehat{\mathbb{Z}})$ (when restricted to $[0,T_0]$). Since $\widehat{\mathbb{Z}}$ is supported in $\mathbb{R}^n\times V_y$, we conclude that the trace of $\mathbf{Y}$ lives indeed in $V_y$. \end{proof}
	We get now the main theorem of this paper.
	\begin{theo}\label{THM:ReachThm} Let $p\geq 1$ and $\gamma>p+1$. Let $n, d\in\mathbb{N}^*$. Let $\mathcal{D}=\{f^1,\ldots, f^n\}$ be a family of Lip-$\gamma$ vector fields on $\mathbb{R}^d$ and define the map $f:\mathbb{R}^d\to \mathcal{L}(\mathbb{R}^n,\mathbb{R}^d)$ by
	\[ \forall y\in \mathbb{R}^d, \forall x=(x_1,\ldots,x_n) \in \mathbb{R}^n: \quad
	f(y)(x)=\sum_{i=1}^n x_if^i(y).
	\] 
	Let $y_0\in \mathbb{R}^d$. Let $\mathbf{X}$ be a geometric $p$-rough path on $\mathbb{R}^n$. Then the solution to the rough differential equation 
	\begin{equation}\label{eq:RDEManifold_3}
	\left\{\begin{array}{l}
	\mathrm{d}\mathbf{Y}_t=f(\mathbf{Y}_t)\mathrm{d}\mathbf{X}_t\\
	\mathbf{Y}_0=y_0\\
	\end{array}
	\right.
	\end{equation}
	takes values in $\mathbb{L}_{y_0}(\mathcal{D})$, the $\mathcal{D}$-orbit at $y_0$. In particular, for every $T\geq 0$, there exists a piecewise linear path $x$ in $\mathbb{R}^d$ such that $y_0+\mathbf{Y}^1_{0,T}=\gamma(y_0,x,T)$, where $(\gamma(y_0,x,t))_{0\leq t \leq T}$ is the solution to the ODE
	\begin{equation}\label{eq:ODEManifold_3}
	\left\{\begin{array}{l}
	\mathrm{d}\gamma_t=f(\gamma_t)\mathrm{d}x_t\\
	\gamma_0=y_0\\
	\end{array}
	\right.
	\end{equation}
	\end{theo}
	\begin{proof} First note that for $z_1, z_2 \in \mathbb{R}^d$, if $V_{z_1}\cap V_{z_2}\neq \varnothing$, then $z_1$ and $z_2$ lie on the same $D$-orbit. Combining this trivial observation with Lemma \ref{lemma:transferSolSubmMan} and a compactness argument yields the first part of the theorem. To prove the second part, we make the observation that all points of $\mathbb{L}_{y_0}(\mathcal{D})$ are solutions to ODEs of the type (\ref{eq:ODEManifold_3}) driven by piecewise linear paths. Indeed, let $y\in\mathbb{R}^d$, $i\in [\![1,n]\!]$, $t\in\mathbb{R}$ and $z=\tilde{f}^i_{t}(y)$. If $t\geq 0$ (resp. $t\leq 0$), then $z$ is the value at time $t$ of the solution to (\ref{eq:ODEManifold_3}) when driven by the linear path $x:\;s\mapsto s\cdot e_i$ (resp. $x:\;s\mapsto -s\cdot e_i$), with $(e_1,\ldots,e_n)$ denoting the canonical basis of $\mathbb{R}^n$. In other words, $z=\gamma(y,x,|t|)$. Note also that the time $|t|$ of the arrival of such solution to $z$ can be arbitrarily chosen by rescaling (time-reparametrising) $x$. More generally, let $z$ be a point in $\mathbb{L}_{y_0}(\mathcal{D})$, $p$ be a positive integer, $f^1,\ldots,f^p$ be vector fields in $\mathcal{D}$ and $t_1,\ldots,t_p$ be real numbers such that
	\[z=\tilde{f}^p_{t_p}\circ\cdots\circ \tilde{f}^1_{t_1}(y_0).\]
	By the argument above, let $x_1$ be a linear path such that 
	\[\tilde{f}^1_{t_1}(y_0)=\gamma\left(y_0,x_1,\frac{|t_1|}{|t_1|+\ldots+|t_p|} T\right)
	.\]
	Inductively, for every $i\in [\![2,p]\!]$, let $x_i$ be a linear path such that 
	\[
	\tilde{f}^i_{t_i}\circ \cdots \circ \tilde{f}^1_{t_1}(y_0)
	=\gamma\left(
	\tilde{f}^{i-1}_{t_{i-1}}\circ \cdots\circ \tilde{f}^1_{t_1}(y_0),x_i,\frac{|t_i|}{|t_1|+\ldots+|t_p|} T
	\right).
	\]
	Then
	\[
	z=\gamma\left(y_0,x_p*\cdots*x_1,T\right),
	\]
	where $x_p*\cdots*x_1$ is the concatenation of $x_1,\ldots,x_p$ (in this order), which is a piecewise linear path. This concludes the proof.
	\end{proof}
	\section{Concluding remarks}\label{sec:Conclusion}	
	In \cite{Sussmann}, Sussmann constructs a smooth structure on orbits by ``patching up'' the smooth structures of the local neighborhoods (denoted $V_y$ in this article, see Theorem \ref{LipCharts}). However, we are not able to build such a global structure in the Lipschitz case due to the rigidity of the current definition of Lipschitz structures. Indeed, these are specifically tailored to allow for global solutions  to RDEs on manifolds to exist, but bridges to and from classical smooth structures are currently lacking (the only non-trivial exception being the case of smooth compact manifolds, see \cite{CLL2}). In the case of orbits, patching up the local neighborhoods yields a locally Lipschitz structure (as defined in \cite{BL}) which allows only for local solutions to RDEs (this result was not proven in \cite{BL} but can be carried using a local argument). It is difficult however to imagine that general RDEs do not admit global solutions on orbits of vector fields since (the traces of) these solutions when considered in the ambient Euclidean space lie in the orbits. But whether the orbits admit a Lipschitz structure or a ``weaker'' structures that still allows for the global solution to RDEs remains an open question to our knowledge.\\
	
	A more important problem -at least from an analytical point of view- is to quantify the necessary length of a piecewise linear path to yield the same terminal solution as a given geometric rough path. If we try and exploit our geometric argument by breaking down the original signal so that the solution is written as a concatenation of rough paths that live in individual charts, then the question boils down to solving the same problem when restricted to the charts $V_y$. However, the understanding of these charts $V_y$ requires a deep knowledge of the whole orbit and on how to transport new tangent directions from points that are possibly very far away from $y$. Consequently, it is more judicious to reserve such argument for special cases where such knowledge is readily available.  An interesting example thereof is the case of the free nilpotent group and the inversion of signatures (e.g. \cite{LX, LX2}).
	
	\bibliographystyle{alpha}
	\bibliography{bibLipOrbit}
\end{document}